\documentclass[11pt, reqno]{amsart}
\usepackage{xr}
\usepackage{amsthm,amssymb,amstext,amscd,amsfonts,soul}
\usepackage{amsbsy,amsxtra,latexsym,url,multirow}
\usepackage{amsmath,color,comment,accents}
\usepackage{fancybox}
\usepackage{fullpage}
\usepackage[english]{babel}
\usepackage[latin1]{inputenc}

\usepackage{float}\restylefloat{table}
\usepackage{bm}
\usepackage{enumitem} 
\usepackage[draft]{hyperref}	

\newcommand{\field}[1]{\mathbb{#1}}
\newcommand{\N}{\field{N}}
\newcommand{\Z}{\field{Z}}%

\newcommand{\C}{\field{C}}
\newcommand{\Q}{\field{Q}}

\newcommand{\SL}{\operatorname{SL}}

\newcommand{\diag}{diag}

\newcommand{\bea}{\begin{eqnarray}}
\newcommand{\eea}{\end{eqnarray}}
\newcommand{\be}{\begin {equation}}
\newcommand{\ee}{\end{equation}}

\newcommand{\ord}{\operatorname{ord}}

\renewcommand{\H}{\mathbb{H}}

\newcommand{\gen}{\text{gen}}
\newcommand{\pgen}{\text{gen}^+}
\newcommand{\spn}{\text{spn}}
\newcommand{\pspn}{\text{spn}^+}
\newcommand{\cls}{\text{cls}}
\newcommand{\pcls}{\text{cls}^+}
\newcommand{\T}[1]{\Theta_{#1}}
\newcommand{\legendre}[2]{\left( \frac{#1}{#2} \right)}

\numberwithin{equation}{section}
\numberwithin{table}{section}
\newtheorem{theorem}{\textbf{Theorem}}
\numberwithin{theorem}{section}
\newtheorem{lemma}[theorem]{\textbf{Lemma}}
\newtheorem{proposition}[theorem]{\textbf{Proposition}}
\newtheorem{remark}[theorem]{Remark}
\newtheorem{remarks}[theorem]{Remarks}
\newtheorem{conjecture}[theorem]{\textbf{Conjecture}}

\newtheorem{corollary}[theorem]{\textbf{Corollary}}

\theoremstyle{remark}

\renewcommand{\pmod}[1]{\, \left(  \mathrm{mod} \,  #1 \right)}

\begin{document}
\title{Theta series of ternary quadratic lattice cosets}
\author{Ben Kane}
\address{Department of Mathematics, University of Hong Kong, Pokfulam, Hong Kong}
\email{bkane@hku.hk}
\author{Daejun Kim}
\address{Department of Mathematics Education, Korea University, Seoul 02841, Republic of Korea}
\email{daejunkim@korea.ac.kr}
\thanks{ The research of the first author was supported by grants from the Research Grants Council of the Hong Kong SAR, China (project numbers HKU 17301317, 17303618, 17307720, and 17314122). The research of the second author was supported by Basic Science Research Program through the National Research Foundation of Korea(NRF) funded by the Minister of Education (NRF-2020R1A6A3A03037816), and by a KIAS Individual Grant (MG085501) at Korea Institute for Advanced Study. Part of the research was conducted while the second author was an honorary research associate at the University of Hong Kong and part was conducted while the second author was a research fellow at the Korea Institute for Advanced Study.}
\date{\today}
\begin{abstract}
In this paper, we consider the decomposition of theta series for lattice cosets of ternary lattices. We show that the natural decomposition into an Eisenstein series, a unary theta function, and a cuspidal form which is orthogonal to unary theta functions correspond to the theta series for the genus, the deficiency of the theta series for the spinor genus from that of the genus, and the deficiency of the theta series for the class from that of the spinor genus, respectively. These three pieces are hence invariants of the genus, spinor genus, and class, respectively, extending known results for lattices and verifying a conjecture of the first author and Haensch. We furthermore extend the definition of $p$-neighbors to include lattice cosets and construct an algorithm to compute respresentatives for the classes in the genus or spinor genus via the $p$-neighborhoods.
\end{abstract}
\maketitle

\section{Introduction and statement of results}

In this paper, we are interested in an interplay between the algebraic and analytic theories of quadratic lattice cosets, which are linked by their theta series, with a particular interest in the ternary case. Let $V$ be a positive definite quadratic space over $\Q$ with the associated non-degenerate symmetric bilinear form
$$
B : V\times V \longrightarrow \Q \quad \text{with} \quad Q(x):=B(x,x)
$$
for any $x\in V$. For a $\Z$-lattice $L$ on $V$ of rank $k$ and a non-zero vector $\nu\in V$, we call $L+\nu$ a {\em lattice coset} or {\em shifted lattice}. If $\nu \in L$, then the lattice coset $L+\nu$ is nothing but the lattice $L$.
By suitable scaling of the quadratic map $Q$, if necessary, we may assume that $Q(L+\nu)\subseteq\Z$.
The theta series $\T{L+\nu}$ of $L+\nu$ is defined to be the generating function for the elements of $L+\nu$ of a given norm, that is, the following function defined on the upper-half complex plane $\mathbb{H}$,
$$
\T{L+\nu}(z)=\sum_{x\in L+\nu} q^{Q(x)}=\sum_{n\ge0} r(n,L+\nu) q^n,
$$
where $r(n,L+\nu):=\lvert\{x\in L+\nu : Q(x)=n\}\rvert$ and $q:=e^{2\pi iz}$ ($z\in \H$). It is well known that $\T{L+\nu}$ is a modular form of weight $k/2$ for some congruence subgroup and a character (for an explicit statement, see Proposition \ref{prop-thetaofcoset-is-a-modularform}). Hence $\T{L+\nu}$ naturally splits into the sum of two pieces; namely, 
\[
\T{L+\nu}=E_{L+\nu}+C_{L+\nu},
\]
where $E_{L+\nu}$ is an Eisenstein series and $C_{L+\nu}$ is a cusp form, and this splitting is unique because it is an orthogonal splitting under the Petersson inner product. Generalizing work of Siegel \cite{SiegelQuadratische} and Weil \cite{Weil} (who considered the $\nu=0$ case), Shimura \cite{ShimuraCongruence} showed that $E_{L+\nu}$ is equal to
\begin{equation}\label{eqn:SiegelAverageGenus}
\T{\gen(L+\nu)}:= \frac{1}{\sum_{K+\mu\in \gen(L+\nu)}o(K+\mu)^{-1}} \sum\limits_{K+\mu\in \gen(L+\nu)}\frac{\T{K+\mu}}{o(K+\mu)},
\end{equation}
where $o(K+\mu)$ is the number of automorphs of the lattice coset, and the sums run over a complete set of representatives of the classes in the genus $\gen(L+\nu)$ of $L+\nu$.

On the other hand, for the ternary case (when $k=3$), the cusp form $C_{L+\nu}$ is further decomposed into two pieces,
\[
C_{L+\nu}=U_{L+\nu}+f_{L+\nu},
\]
where $U_{L+\nu}$ is in the space of unary theta functions and $f_{L+\nu}$ is a cusp form orthogonal to unary theta functions with respect to the Petersson inner product. In the case of lattices, Schulze-Pillot \cite{Schulze-PillotTernaryTheta} showed that one may isolate the unary theta functions in this decomposition by taking a weighted average analogous to \eqref{eqn:SiegelAverageGenus}, with the sum instead running over classes of the spinor genus of the associated lattice.

Motivated by Schulze-Pillot's result and examples that resolved questions related to representations of sufficiently large integers by lattice cosets, Haensch and the first author \cite[Conjecture 1.3]{HaenschKane} conjectured that the same decomposition holds for lattice cosets. Namely, setting\footnote{We often distinguish between the genus (resp. spinor genus) and the proper genus (resp. proper spinor genus), adding a $+$ to the notation when investigating the proper classes.}
\[
\T{\pspn(L+\nu)}:= \frac{1}{\sum_{K+\mu\in \pspn(L+\nu)}o^+(K+\mu)^{-1}} \sum\limits_{K+\mu\in \pspn(L+\nu)}\frac{\T{K+\mu}}{o^+(K+\mu)},
\]
where the sum runs over a complete set of representatives of the proper classes in the proper spinor genus $\pspn(L+\nu)$ of $L+\nu$ and $o^+(K+\mu)$ is the number of proper automorphs of the lattice coset (we refer the reader to Section \ref{subsection-prelim-lattice cosets} for the definition of the proper genus $\pgen(L+\nu)$, the proper spinor genus $\pspn(L+\nu)$, and the proper class $\pcls(L+\nu)$  of $L+\nu$), they conjectured the following.
\begin{conjecture}\label{conj:HaenschKane}
For a quadratic lattice $L$ and $\nu\in\Q L$, we have
\[
\T{\pspn(L+\nu)}= E_{L+\nu} +\mathcal{U}_{\pspn(L+\nu)},
\]
where $\mathcal{U}_{\pspn(L+\nu)}$ is a linear combination of unary theta functions.
\end{conjecture}

In this paper, we prove that Conjecture \ref{conj:HaenschKane} is true, with $\mathcal{U}_{\pspn(L+\nu)}=U_{L+\nu}$,  and obtain a dictionary between natural objects occuring in the algebraic theory of lattice cosets and the orthogonal projections of $\T{L+\nu}$ into the subspaces of Eisenstein series, unary theta functions, and cusp forms orthogonal to unary theta functions. Let $L+\nu$ be a ternary lattice coset and consider the natural splitting of its theta series
\begin{equation}\label{eqn-splitting-of-thetaofcosets}
	\begin{array}{rcccccc}
		\T{L+\nu}&=&\T{\pgen(L+\nu)}&+& (\T{\pspn(L+\nu)}-\T{\pgen(L+\nu)}) &+& (\T{L+\nu}-\T{\pspn(L+\nu)})\\
		&=&E_{L+\nu} &+& U_{L+\nu} &+& f_{L+\nu}.
	\end{array}	
\end{equation}
Here the theta series $\T{\pgen(L+\nu)}$ and $\T{\pspn(L+\nu)}$ are defined as \eqref{defn-thetaofpropergenus} and \eqref{defn-thetaofproperspinorgenus} (or the above), respectively. Our main result is that the two splittings of $\T{L+\nu}$ in \eqref{eqn-splitting-of-thetaofcosets} indeed coincide termwise.
\begin{theorem}\label{thm:main}
Conjecture \ref{conj:HaenschKane} is true. Moreover, the following hold:
\begin{enumerate}[label={\rm (\arabic*)}]
	\item $E_{L+\nu}(z)=\T{\pgen(L+\nu)}(z)$ is an invariant of $\pgen(L+\nu)$ (Corollary \ref{cor-theta-pgen-gen-equal}),
	\item $U_{L+\nu}(z)=\T{\pspn(L+\nu)}(z)-\T{\pgen(L+\nu)}(z)$ is an invariant of $\pspn(L+\nu)$ (Theorem \ref{thm-diff-thetaofpspn-thetaofpgen}),
	\item $f_{L+\nu}(z)=\T{L+\nu}(z)-\T{\pspn(L+\nu)}(z)$ is an invariant of $\pcls(L+\nu)$ (Theorem \ref{thm-Theta-same-spn}).
\end{enumerate}
\end{theorem}

\begin{remarks}\label{rem:AfterMain}
\noindent

\noindent
\begin{enumerate}[leftmargin=*,label={\rm (\arabic*)}]
\item As noted above, we scale our lattice cosets so that they are integral. Hence in Corollary \ref{cor-theta-pgen-gen-equal}, Theorem \ref{thm-diff-thetaofpspn-thetaofpgen}, and Theorem \ref{thm-Theta-same-spn}, $L+\nu$ is replaced with $aL+\nu$ with $\nu\in L$, where this scaling is done so that we may start with an arbitrary integral lattice $L$. In order to translate these theorems into the forms listed above, see the definitions in Section \ref{section-preliminary} for the precise setting.
\item
 Note that if $\nu\notin L$, then $\gen(L+\nu)$ may not coincide with $\pgen(L+\nu)$ (for example, see \cite[Example 4.5]{ChanOh13}). However, our result combined with Shimura's result in \cite{ShimuraCongruence} on the Eisenstein series $E_{L+\nu}$ implies that $\T{\gen(L+\nu)}(z)=\T{\pgen(L+\nu)}$ when $k=3$. We expect that $\T{\gen(L+\nu)}=\T{\pgen(L+\nu)}$ for $k\ge 4$, but we do not investigate this question here and are not aware of a proof in the general case.

\item
 As noted above, if $\nu\in L$, then the splitting \eqref{eqn-splitting-of-thetaofcosets} of $\T{L+\nu}=\T{L}$ was obtained in previous work of Schulze-Pillot \cite{Schulze-PillotTernaryTheta} for ternary lattices. 
In fact, for a ternary lattice $L$, we know that $\gen(L)=\pgen(L)$, $\spn(L)=\pspn(L)$, and $\cls(L)=\pcls(L)$ because $-1_V$ is an automorph of $L$ with determinant $-1$. Hence the theta series of the proper genus and the spinor genus coincide with that of the genus and spinor genus, respectively. Schulze-Pillot determined $U_L$ from algebraic properties of the $p$-neighborhood of ternary lattices.
In Section \ref{subsection-p-nbd-cosets}, we extend the concept of $p$-neighborhoods of ternary lattices to that of ternary lattice cosets, and study some algebraic properties and their interplay with the Hecke operators on $\T{L+\nu}$. We also provide a way to explicitly determine $U_{L+\nu}$ (see Corollary \ref{cor-determining-coefficientsofunarytheta}).

\item There is a natural connection between lattice cosets and quadratic forms with congruence conditions, so Theorem \ref{thm:main} yields a natural splitting for theta functions of quadratic forms with congruence conditions. Along this vein, Duke and Schulze-Pillot \cite{DukeSchulzePillot} proved a similar statement with a modified definition of congruence class (genus, spinor genus) modulo $N(\in\N)$ that agrees with ours in the case of lattices. Their definitions of congruence class and genus coincide with that of van der Blij \cite{vanderBlij}, who proved the Siegel--Weil formula for quadratic forms with congruence conditions.
We note that although the definitions for these algebraic objects are different, their corresponding theta series should coincide because the splitting $\T{L+\nu}=E_{L+\nu}+U_{L+\nu}+f_{L+\nu}$ is unique (see \cite[Lemma 4]{DukeSchulzePillot}). Methods for computing the congruence classes in the congruence genus (or in the congruence spinor genus) have also not been studied, as far as the authors know, but in Section \ref{subsection-graph} we use an object constructed to prove Theorem \ref{thm:main} (3) to design an algorithm that returns a full set of representatives of the proper (spinor) genus in our setting. 

\end{enumerate}
\end{remarks}

In order to prove Theorem \ref{thm:main}, we investigate the action of the Hecke operators on theta series of lattice cosets in Theorem \ref{thm-Hecke-actonthetaofcoset} and Theorem \ref{thm-Hecke-nbd}. Defining the \begin{it}conductor\end{it} of $L+\nu$ to be the minimal $a$ such that $a\nu\in L$, the action of the Hecke operators reveal a connection between $L+\nu$ and other lattice cosets $K+\mu$ with the same conductor. As a side-effect, we establish a definition of $p$-neighborhoods of shifted lattices (see Section \ref{subsection-p-nbd-cosets}); in the case of lattices, these $p$-neighborhoods have played an important role in explicit constructions of the genus and spinor genus (see \cite{Schulze-PillotAlgorithm}), a task which has previously proven difficult for shifted lattices. After establishing these connections, most of the results involving the theta series of the spinor genus can be obtained via measure-theoretic results already in the literature, up to a few tricky technical details that arise from the relations between shifted lattices coming from the same initial lattice. 

The splitting \eqref{eqn-splitting-of-thetaofcosets} of $\T{L+\nu}$ is also useful for determining which sufficiently-large positive integers are represented by $L+\nu$ and it gives an asymptotic formula for $r(n,L+\nu)$. The $n$-th Fourier coefficient of the Eisenstein series $E_{L+\nu}$ is kind of explicit in the sense that one may write it as a product of local representation densities, using the Siegel--Weil formula for lattice cosets proved by Shimura \cite{ShimuraCongruence}. Moreover, as long as $n$ goes to infinity with bounded divisibility by certain (finitely many) ``bad'' primes and $n$ is locally represented (i.e., there are no obstructions coming from congruence conditions), it grows at least like $n^{1/2-\varepsilon}$.
The Fourier coefficients of $U_{L+\nu}$ also grow as fast as that of $E_{L+\nu}$, but these are sparse; namely, the coefficients are supported on finitely many square classes (see Theorem \ref{thm-diff-thetaofpspn-thetaofpgen}). Furthermore, one may explicitly determine $U_{L+\nu}$ by computing only finitely many coefficients of $\T{\pspn(L+\nu)}-\T{\pgen(L+\nu)}$ (see Corollary \ref{cor-determining-coefficientsofunarytheta}). On the other hand, a result of Duke \cite{DukeHyperbolic} implies that the absolute value of the $n$-th Fourier coefficient of $f_{L+\nu}$ grows at most like $n^{3/7+\varepsilon}$, and hence the contribution from this term may generally be considered to be an error term. Therefore, every sufficiently-large positive integer $n$ which has bounded divisibility at the ``bad'' primes, is locally represented by $L+\nu$, and does not belong to any of the finitely many exceptional square classes is represented by $L+\nu$.

For a given shifted lattice $L+\nu$, one can naively obtain the splitting $\T{L+\nu}=E_{L+\nu}+U_{L+\nu}+f_{L+\nu}$ by constructing a basis of the corresponding space of modular forms and applying linear algebra directly. However, the dimension of the space grows somewhat quickly with respect to the discriminant of the lattice and the conductor of the shifted lattice, so this is only practical for relatively small discriminants and conductors. Our result circumvents the need to do high-dimensional linear algebra, yielding an independent algorithm for computations. This algorithm requires only the construction of a system of representatives of the proper classes of $\pgen(L+\nu)$ and $\pspn(L+\nu)$. 
Using our modification of the definition of the $p$-neighborhood of a lattice to that of lattice cosets, there is an algorithmic way, at least in principle, to list out the representatives, generalizing the algorithm in \cite{Schulze-PillotAlgorithm} for finding representatives in the case of lattices (see Section \ref{section-thetainthesamespinorgenus} for further details).

The paper is organized as follows. We first give some preliminary definitions and known results in Section \ref{section-preliminary}.
Especially, the particular space of modular forms in which the theta series $\T{L+\nu}$ lies is described, and the Hecke operators are defined.
In Section \ref{section-algebraicstructureofcosets}, we introduce some algebraic structure of lattice cosets including $p$-neighborhoods of lattice cosets.
In Section \ref{section-Hecke-on-theta}, we discuss how the action of the Hecke operators on the theta series of lattice cosets is related to its $p$-neighborhood, and determine the Eisenstein series $E_{L+\nu}$.
We investigate $U_{L+\nu}$ in Section \ref{section-unarythetafunctionpart}, and we finally determine that $f_{L+\nu}$ is orthogonal to unary theta functions in Section \ref{section-thetainthesamespinorgenus}.

\section{Preliminaries}\label{section-preliminary}

\subsection{Quadratic lattice cosets}\label{subsection-prelim-lattice cosets}
We introduce some definitions of quadratic spaces, lattices, and lattice cosets, and describe our setting for lattice cosets. We refer readers to \cite{OMBook} for more details.

As in the introduction, let  $V$ be a positive definite quadratic space over $\Q$ with the associated non-degenerate symmetric bilinear form
$$
B : V\times V \longrightarrow \Q \quad \text{with} \quad Q(x)=B(x,x)
$$
for any $x\in V$ and the special orthogonal group 
$$
O^+(V)=\{\sigma\in GL(V) : B(\sigma x,\sigma y)=B(x,y) \text{ for any }x,y\in V \text{ and } \det(\sigma)=1 \}.
$$
Let $\theta:O^+(V)\rightarrow \Q^\times/(\Q^\times)^2$ be the {\em spinor norm map} (cf. \cite[$\S 55$]{OMBook}) and denote its kernel by
$$
O'(V)=\{\sigma\in O^+(V) : \theta(\sigma)=1\}.
$$
Let $O_A^+(V)$  and $O_A'(V)$ be the {\em ad{\'e}lizations} of $O^+(V)$ and $O'(V)$, respectively.

A finitely-generated $\Z$-module (hence free $\Z$-module) $L$ in $V$ is called be a {\em lattice} on $V$ if $\Q L=V$.
Let $\Omega=\Omega_\Q$ be the set of all spots (or places) including the infinite spot $\infty$.
We denote the {\em localization} of a lattice $L$ in the localization $V_p$ of $V$ at $p$ by $L_p$ for any prime spot $p$ and $L_\infty = V_\infty$.

Consider a lattice $L$ on $V$. For any non-zero vector $v_0\in V$, we define the {\em shifted lattice} in $V$ to be the set $L+v_0$.
The {\em conductor} of a shifted lattice $L+v_0$ is defined to be the smallest positive integer $a$ such that $av_0\in L$. We can always realize a quadratic Diophantine equation as being induced from a shifted lattice in some quadratic space (see Section 1 of \cite{LSun}).
This is equivalent to study the set $aL+\nu$ (where $\nu=av_0$) in $V$, which is a coset in $L/aL$.
Hence, throughout this article, the term ``lattice coset", or simply ``coset", always refers to the set $aL+\nu$, where $L$ is a lattice on $V$, $a$ is a positive integer, and $\nu\in L$ whose conductor with respect to $aL$ is equal to $a$. This is to emphasize the role of the conductor of lattice cosets in our results.

We always assume that any lattice $L$ is {\em integral}, that is, $B(L,L)\subseteq \Z$ so that we have $Q(aL+\nu)\subseteq \Z$. The {\em discriminant} $d_L$ of $L$ is the determinant of the matrix $A=(B(e_i,e_j))$ for a basis $\{e_i\}$ of $L$, and the {\em level} $N_L$ is defined to be the smallest positive integer $N$ such that $NA^{-1}$ has coefficients in $\Z$.

From \cite[Lemma 4.2]{ChanOh13} or \cite[Lemma 1.2]{LSun}, $O_A^+(V)$, $O^+(V)$, and $O_A'(V)$ all act on $aL+\nu$. Hence we may define
$$
\pgen(aL+\nu):=\text{the orbit of } aL+\nu \text{ under the action of }O_A^+(V)
$$
which is called the {\em proper genus} of $aL+\nu$,
$$
\pspn(aL+\nu):=\text{the orbit of } aL+\nu \text{ under the action of }O^+(V)O_A'(V)
$$
which is called the {\em proper spinor genus} of $aL+\nu$, and
$$
\pcls(aL+\nu):=\text{the orbit of } aL+\nu \text{ under the action of }O^+(V)
$$
which is called the {\em proper class} of $aL+\nu$. Clearly,
$$
\pcls(aL+\nu)\subseteq \pspn(aL+\nu) \subseteq \pgen(aL+\nu).
$$
Set
$$
O^+(aL+\nu)=\{\sigma\in O^+(V) : \sigma(aL+\nu)=aL+\nu\} \quad \text{and} \quad o^+(aL+\nu)=|O^+(aL+\nu)|.
$$
The groups $O^+(aL_p+\nu)$ for any prime $p$ and $O_A^+(aL+\nu)$ may be defined analogously.

The number of (proper) classes in $\pgen(aL+\nu)$ is called the {\em class number} of $aL+\nu$. It is well-known that the class number is equal to the number of double cosets in $O^+(V)\backslash O_A^+(V)/O_A^+(aL+\nu)$ and this is finite (see \cite[Corollary 2.3]{LSun}, see also \cite[Corollary 4.4]{ChanOh13}). The number $g^+(aL+\nu)$ analogously counts the number of (proper) spinor genera contained in the (proper) genus. The next proposition recalls and extends \cite[Proposition 2.5]{Xu}.

\begin{proposition}\label{prop-spinornormmap}
	Let $aL+\nu$ be a coset on a quadratic space $V$ over $\Q$, and let $\theta$ be the spinor norm map defined on $O_A^+(V)$.
	If $\dim(V)\ge 3$, then the number of proper spinor genera in $\pgen(aL+\nu)$ is given by
	$$
	[I_\Q : \Q^\times \prod\limits_{p\in \Omega} \theta(O^+(aL_p+\nu))].
	$$
	Moreover, suppose that $\dim(V)= 3$ and let $x\in V$ be a non-zero vector with $Q(x)=n$ and $V=\Q x\perp W$. Then the spinor norm map induces an isomorphism
	$$
	O_A^+(V)/O^+(V)O_A'(V)O_A^+(W)O_A^+(aL+\nu) \cong I_\Q / \Q^\times N_{E/\Q}(I_E)\prod\limits_{p\in \Omega} \theta(O^+(aL_p+\nu)),
	$$
	where $E=\Q(\sqrt{-nd_L})$, $I_\Q$ and $I_E$ are the id{\`e}le groups, and $N_{E/\Q}$ is the norm map.
\end{proposition}
\begin{proof}
	The first assertion was made in \cite[Proposition 2.5]{Xu}, but we provide a brief proof for completeness. Note that for a $u,v\in O_A^+(V)$, the coset $v(aL+\nu)$ belongs to $\pspn(u(aL+\nu))$ if and only if $v\in O^+(V)O_A'(V)uO_A^+(aL+\nu)$.
	This group is equal to $O^+(V)O_A'(V)O_A^+(aL+\nu)u$ since $O_A'(V)$ contains the commutator subgroup of $O_A^+(V)$.
	Hence, the number of proper spinor genera is given by 
	$$
	[O_A^+(V):O^+(V)O_A'(V)O_A^+(aL+\nu)].
	$$
	On the other hand, by \cite[102:7]{OMBook}, the spinor norm map $\theta$ induces the isomorphism
	\begin{equation}\label{bijection-numberofspinorgenera}
	O_A^+(V)/O^+(V)O_A'(V)O_A^+(aL+\nu) \xrightarrow{\cong} I_\Q/\Q^\times \prod\limits_{p\in \Omega} \theta(O^+(aL_p+\nu)).
	\end{equation}
	Furthermore, we show that the map $\theta$ induces the following isomorphism
	\begin{equation}\label{bijection-splittingintotwohalves}
	\theta : O_A^+(V)/O^+(V)O_A'(V)O_A^+(W)O_A^+(aL+\nu) \xrightarrow{\cong} I_\Q/\Q^\times N_{E/\Q}(I_E) \prod\limits_{p\in \Omega} \theta(O^+(aL_p+\nu)).
	\end{equation}
	We first note that $-nd_L$ is not a square in $\Q$ since it is a negative number so that $\theta(O^+(W_p))=N_{E_\mathfrak{p}/\Q_p}(E_\mathfrak{p}^\times)$ for any $\mathfrak{p}\mid p$.
	Hence, the map in \eqref{bijection-splittingintotwohalves} is well-defined. The surjectivity of \eqref{bijection-splittingintotwohalves} follows from that of \eqref{bijection-numberofspinorgenera}.
	Finally, assume that a $s=(s_p)\in O_A^+(V)$ satisfies $\theta(s)=b\cdot j \cdot i$ for some $b\in \Q^\times$, $j=(j_p)\in N_{E/\Q}(I_E)$, and $i=(i_p)\in \prod_{p\in \Omega} \theta(O^+(aL_p+\nu))$. Since all the $\theta(s_\infty)$, $j_\infty$, and $i_\infty$ are positive numbers, we should have $b>0$. 
	Thus, $b=\theta(\sigma)$ for some $\sigma\in O^+(V)$ by \cite[101:8]{OMBook}.
	On the other hand, $i_p=\theta(\Sigma_p)$ and $j_p=\theta(h_p)$ for some $\Sigma_p\in O^+(aL_p+\nu)$ and $h_p\in O^+(W_p)$ for any $p\in \Omega$.
	Since $\theta(s_p)=\theta(\sigma)\theta(h_p)\theta(\Sigma_p)$ for any $p\in \Omega$, we may conclude that 
	$$
	s\in \sigma\cdot  h\cdot \Sigma \cdot O_A'(V) \subseteq O^+(V)O_A'(V)O_A^+(W)O_A^+(aL+\nu),
	$$
	where $h=(h_p)\in O_A^+(W)$ and $\Sigma=(\Sigma_p)\in O_A^+(aL+\nu)$. Thus the map in \eqref{bijection-splittingintotwohalves} is injective. This completes the proof of the proposition.
\end{proof}

\subsection{Modular forms}
We briefly introduce modular forms of half-integral weight below. We refer readers to \cite{OnoBook} for an introduction to modular forms of integral weight or for more details.

For a positive integer $N$, we require natural congruence subgroups of $\SL_2(\Z)$ defined by 
$$
\begin{aligned}
	&\Gamma_0(N)=\left\{\left(\begin{smallmatrix} a&b\\c&d \end{smallmatrix}\right) \in \SL_2(\Z) : c \equiv 0 \pmod{N} \right\},\\
	&\Gamma_1(N)=\left\{\left(\begin{smallmatrix} a&b\\c&d \end{smallmatrix}\right) \in \Gamma_0(N) : a,d \equiv 1 \pmod{N} \right\}.
\end{aligned}
$$
For a $\gamma=\left(\begin{smallmatrix} a&b\\c&d \end{smallmatrix}\right)\in \Gamma_0(4)$ and $\kappa\in \frac{1}{2}+\Z$, define the {\em slash operator} on a function $f:\H \rightarrow \C$ by
$$
f\mid_{\kappa}\gamma (z)= \legendre{c}{d}\varepsilon_d^{2\kappa}(cz+d)^{-\kappa}f(\gamma z),
$$
where $\varepsilon_d=1$ if $d\equiv 1 \pmod{4}$, $\varepsilon_d=i$ if $ d\equiv 3 \pmod{4}$, and $\legendre{\cdot}{\cdot}$ is the Kronecker--Jacobi--Legendre symbol. We call $f$ a {\em (holomorphic) modular form} of weight $\kappa$ on $\Gamma\subseteq \Gamma_0(4)$ ($\Gamma$ a congruence subgroup containing $\left(\begin{smallmatrix} 1&1\\0&1 \end{smallmatrix}\right)$) with character $\chi$ if
\begin{enumerate}[label={\rm (\arabic*)}]
	\item $f|_\kappa\gamma = \chi(d)f$ for any  $\gamma=\left(\begin{smallmatrix} a&b\\c&d \end{smallmatrix}\right)\in \Gamma$,
	\item $f$ is holomorphic on $\H$,
	\item $f(z)$ grows at most polynomially in $y$ as $z=x+iy\rightarrow \Q\cup \{i\infty\}$.
\end{enumerate}
We moreover call $f$ a {\em cusp form} if $f(z)\rightarrow 0$ as $z\rightarrow \Q\cup \{i\infty\}$.
The space of modular forms (resp. cusp forms) of weight $\kappa$, character $\chi$ and congruence subgroup $\Gamma$, will be denoted by $M_{\kappa}(\Gamma,\chi)$ (resp. $S_{\kappa}(\Gamma,\chi)$).
The space of {\em Eisenstein series}, denoted by $E_{\kappa}(\Gamma,\chi)$, is the orthogonal complement of $S_{\kappa}(\Gamma,\chi)$ in $M_{\kappa}(\Gamma,\chi)$ with respect to the Petersson inner product (for an introduction and properties of the inner product, see \cite[Chapter III]{Lang}). If $f$ is a modular form for a congruence group $\Gamma$ containing $\left(\begin{smallmatrix} 1&1\\0&1\end{smallmatrix}\right)$, then $f$ has a Fourier series expansion 
$$
f(z)=\sum\limits_{n=0}^\infty a(n)q^n,
$$
where $q=e^{2\pi iz}$. In particular, if $f$ is a cusp form, then $a(0)=0$.

For $\kappa\ge 3/2$, let $f(z)=\sum\limits_{n=1}^\infty a(n)q^n\in S_{\kappa}(\Gamma_0(N),\chi)$ ($4\mid N$). For a square-free positive integer $t$, define the {\em $t$-th Shimura lift} by
$$
F_t(z)=\sum\limits_{n=1}^\infty A_t(n)q^n,
$$
where $A_t(n)$ is defined by 
$$
\sum\limits_{n=1}^\infty A_t(n)n^{-s}=\left(\sum\limits_{m=1}^\infty \chi(m)\legendre{-1}{m}^{\kappa-\frac{1}{2}}\legendre{t}{m} m^{\kappa-\frac{3}{2}-s}\right)\left(\sum\limits_{m=1}^\infty a(tm^2)m^{-s}\right).
$$
Shimura \cite{Shimura} proved that $F_t(z)\in M_{2\kappa-1}(\Gamma_0(N_t),\chi^2)$ for a suitable $N_t$. Later, Niwa \cite{Niwa} showed that $N_t$ can be taken as $N/2$ independently of $t$.
For $\kappa\ge 5/2$, $F_t$ is a cusp form, but the situation is more complicated when $\kappa=3/2$, requiring a more careful analysis of the space  
$$
U_t(N,\chi):=S_{3/2}(\Gamma_0(N),\chi) \cap \left\{f(z)=\sum\limits_{n=1}^\infty a(n)nq^{tn^2}\right\}
$$
spanned by unary theta functions. Specifically from the results in \cite{Cipra}, \cite{Kojima}, \cite{Sturm}, the $t$-th Shimura lift $F_t$ of $f$ is a cusp form if and only if $f$ belongs to the orthogonal complement $U_t^\perp$ of $U_t$ in $S_{3/2}(\Gamma_0(N),\chi)$ with respect to the Petersson inner product. For a Dirichlet character $\psi$ modulo $m_\psi$, consider
\[
 h(z,\psi)=\sum_{n=1}^\infty \psi(n)nq^{n^2}.
\]
Note that the space $U_t(N,\chi)$ is spanned by 
\begin{equation}\label{UtNchi-is-spanned-by}
	\left\{h(tu^2z,\psi) : u\in\Z, \ 4tm_\psi^2u^2\mid N, \ \psi = \chi \legendre{-4t}{\cdot}\right\}.
\end{equation}
This follows from the fact that the spaces $h(tu^2z,\psi)$ for different $t$ or $\psi$ are orthogonal to each other with respect to Petersson inner product and the modularity given in \cite[Proposition 2.2]{Shimura}.

\subsection{Elementary theta functions} Let $k$ be a positive integer, $A$ a positive definite $k\times k$ symmetric matrix, $h$ an element in $\Z^k$, and $N$ a positive integer satisfying the following conditions:
\begin{equation}\label{condition-thetafunctionofShimura}
	\text{Both $A$ and $NA^{-1}$ have coefficients in $\Z$; } Ah\in N\Z^k.
\end{equation}
In \cite{Shimura}, Shimura defined the theta function 
\begin{equation}\label{defn-thetafunctionofShimura}
	\vartheta(z;h,A,N,P)=\sum\limits_{x\in\Z^k,\, x\equiv h\pmod{N}} P(x) \cdot q^{ x^tAx/2N^2},
\end{equation}
where $P(x)$ is a spherical function of order $\nu\in\Z_{\ge 0}$ with respect to $A$. In this article, we only concern the case when $P(x)=1$ where $\nu=0$, or $P(x)=x$ with $k=1$ where $\nu=1$. 
Indeed, the function $h(z,\psi)$ defined in \eqref{UtNchi-is-spanned-by} is given by a linear combination of the theta functions corresponding to the latter case (see \cite[Proposition 2.2]{Shimura}).
Moreover, Shimura \cite{Shimura} proved the following transformation formula of the theta functions.
\begin{proposition}[Proposition 2.1 of \cite{Shimura}]\label{prop-Shimura-transformation-theta}
	Let $\vartheta(z;h,A,N,P)$ be defined by \eqref{defn-thetafunctionofShimura} under the assumption \eqref{condition-thetafunctionofShimura}, and let $\gamma=\left(\begin{smallmatrix}
		a&b\\c&d
	\end{smallmatrix}\right)\in \SL_2(\Z)$ with $b\equiv 0\pmod{2}$ and $c\equiv 0 \pmod{2N}$. Then
$$
\vartheta(\gamma(z);h,A,N,P)=e(ab\cdot h^tAh/2N^2)\legendre{\det(A)}{d}\legendre{2c}{d}^k\varepsilon_d^{-k}(cz+d)^{(k+2\nu)/2}\vartheta(z;ah,A,N,P),
$$
where $e(z):=e^{2\pi i z}$.
\end{proposition}

\subsection{Masses of genera and spinor genera}

 For $\mathfrak{h}=\pspn(aL+\nu)$ or $\pgen(aL+\nu)$, define the \begin{it}mass of $\mathfrak{h}$\end{it} by
\begin{equation}\label{eqn:Mass}
\text{Mass}(\mathfrak{h}):= \sum_{aK+\mu\in\mathfrak{h}} \frac{1}{o^+(aK+\mu)},
\end{equation}
where the sum runs over a system of proper classes in $\mathfrak{h}$. Using Lemma \ref{lem-measure-repnum} and \cite[Corollary 2.5]{LSun}, one may relate the masses of the proper spinor genus and the proper genus of a shifted lattice via 
$$
\begin{aligned}
\text{Mass}(\pspn(aL+\nu))&=\frac{1}{g^+(aL+\nu)}\text{Mass}(\pgen(aL+\nu))\\
&= \text{Mass}(\pgen(aL))\frac{\prod_{p<\infty} [O^+(aL_p):O^+(aL_p+\nu)]}{g^+(aL+\nu) },
\end{aligned}
$$
where $g^+(aL+\nu)$ is the number of proper spinor genera in $\pgen(aL+\nu)$. Generally speaking, each of the factors on the right-hand side of the above equation may be explicitly computed; $\text{Mass}(\pgen(aL))$ may be deterimined via the Minkowski--Siegel formula and for almost all prime $p$ we have $[O^+(aL_p):O^+(aL_p+\nu)]=1$, while these indices can be computed in general. Based on work of Xu \cite[Proposition 2.5]{Xu}, a formula for $g^+(aL+\nu)$ is given in Proposition \ref{prop-spinornormmap} and in practice one can evaluate the quantities there, although a general formula is not known. 

In Section \ref{section-thetainthesamespinorgenus}, we develop an algorithm for computing the representatives of the proper (spinor) genus of a shifted lattice, which could in principle be used to directly compute \eqref{eqn:Mass} via the definition, after an appropriate calculation of $o^+(aK+\mu)$.

\subsection{Theta series for cosets}\label{subsection-prelim-thetaseriesofcosets} Let $aL+\nu$ be a coset on a quadratic space $V$ of rank $k$.
Note that $Q(aL+\nu)\subseteq \Z$ since we are assuming $B(L,L)\subseteq \Z$.
For a positive integer $n$, we define
$$
R(n,aL+\nu):=\{x\in aL+\nu : Q(x)=n\} \quad \text{and} \quad r(n,aL+\nu):= |R(n,aL+\nu)|,
$$
and the theta series $\T{aL+\nu}(z)$ of the coset $aL+\nu$ is defined as
$$
\T{aL+\nu}(z):=\sum\limits_{x\in aL+\nu} q^{Q(x)} = \sum\limits_{n=0}^\infty r(n,aL+\nu) q^n.
$$
Note that any coset in $\pgen(aL+\nu)$ has conductor $a$. We define the theta series $\T{\pgen(aL+\nu)}(z)$ of $\pgen(aL+\nu)$ and $r(n,\pgen(aL+\nu))$ by
\begin{equation}\label{defn-thetaofpropergenus}
\begin{aligned}
	\T{\pgen(aL+\nu)}(z)&=\sum\limits_{n=0}^\infty r(n,\pgen(aL+\nu))\cdot q^n\\
&:=\frac{1}{\text{Mass}(\pgen(aL+\nu))}
\left(\sum\limits_{aK+\mu\in\pgen(aL+\nu)}\frac{\T{aK+\mu}(z)}{o^+(aK+\mu)}\right)
\end{aligned}
\end{equation}
and the theta series $\T{\pspn(aL+\nu)}(z)$ of $\pspn(aL+\nu)$ and $r(n,\pspn(aL+\nu))$ by
\begin{equation}\label{defn-thetaofproperspinorgenus}
\begin{aligned}
	\T{\pspn(aL+\nu)}(z)&=\sum\limits_{n=0}^\infty r(n,\pspn(aL+\nu))\cdot q^n \\
	&:= 
\frac{1}{\text{Mass}(\pspn(aL+\nu))}
 \left(\sum\limits_{aK+\mu\in\pspn(aL+\nu)}\frac{\T{aK+\mu}(z)}{o^+(aK+\mu)}\right).
\end{aligned}
\end{equation}
The summation runs over a system of representatives of proper classes in the proper genus or in the proper spinor genus of $aL+\nu$.

For any non-zero integer $d$, let $\chi_d$ denote the character $\chi_d(\cdot)=\legendre{d}{\cdot}$ obtained from the Kronecker symbol.
The following proposition shows that the theta series of cosets of rank $k$ are modular forms of weight $k/2$.
\begin{proposition}\label{prop-thetaofcoset-is-a-modularform}
	Let $aL+\nu$ be a coset on a quadratic space $V$ of rank $k$. Let $N_L$ be the level of $L$ and $d_L$ the discriminant of $L$. Then 
	$$
	\T{aL+\nu}(z)\in 
	\begin{cases} 
		M_{k/2}(\Gamma_0(4N_La^2)\cap \Gamma_1(a), \chi_{4d_L}) & \text{if } k \text{ is odd},\\ 
		M_{k/2}(\Gamma_0(4N_La^2)\cap \Gamma_1(a), \chi_{(-1)^{k/2}4d_L}) & \text{if } k \text{ is even}.
	\end{cases}
	$$
\end{proposition}
\begin{proof}
Let $L=\Z e_1 + \cdots + \Z e_k$, $A$ the Gram matrix of $L$ with respect to the basis $\{e_i\}$, and let $v=(v_1,\ldots, v_k)^t$ where $\nu=v_1e_1+\cdots v_ke_k$ for some $v_i\in\Z$. We abbreviate $N:=N_L$ for ease of notation. Note that both $aA$ and $(Na)^{-1}(aA)=N^{-1}A$ have coefficients in $\Z$, and $(aA)(Nv)\in (aN)\Z^k$.
Moreover,
$$
\vartheta(2az;Nv,aA,Na,1)=\sum\limits_{x\in\Z^k, \, x\equiv Nv \pmod{Na}}  q^{2a \cdot x^t (aA) x /2(Na)^2} =\sum\limits_{x\in\Z^k, \, x\equiv v \pmod{a}}  q^{x^t A x} = \T{aL+\nu}(z).
$$
For any $\gamma=\left(\begin{smallmatrix} p&q\\ r&s\end{smallmatrix}\right)\in \Gamma_0(4N a^2)$, the matrix $\gamma'=\left(\begin{smallmatrix} p& 2aq\\ r/2a&s\end{smallmatrix}\right)\in \Gamma_0(2Na)$, and note that $2a(\gamma z) = \gamma'(2az)$. By Proposition \ref{prop-Shimura-transformation-theta}, we have 
$$
\begin{aligned}
	\T{aL+\nu}&(\gamma z)= \vartheta(2a(\gamma z);Nv,aA,Na,1)=\vartheta(\gamma' (2az);Nv,aA,Na,1)\\
	&=e\left(\frac{p(2aq)N^2aQ(\nu)}{2(Na)^2}\right)\legendre{a^k\det(A)}{s}\legendre{2(r/2a)}{s}^k\varepsilon_s^{-k}(rz+s)^{k/2}\vartheta(2az;pNv,aA,Na,1)\\
	&=\legendre{d_L}{s}\legendre{r}{s}^k\varepsilon_s^{-k}(rz+s)^{k/2}\T{aL+p\nu}(z).
\end{aligned}
$$
Hence we have obtained for any $\gamma=\left(\begin{smallmatrix} p&q\\ r&s\end{smallmatrix}\right)\in \Gamma_0(4N_La^2)$  that
\begin{equation}\label{eqn-theta-transformed-by-Gamma0}
	(\T{aL+\nu}|_{k/2}\gamma)(z) = \chi(s)\T{aL+p\nu}(z)=\chi(\gamma)\T{aL+p\nu}(z),
\end{equation}
where $\chi=\chi_{4d_L}$ if $k$ is odd, and $\chi=\chi_{(-1)^{k/2}4d_L}$ if $k$ is even. Furthermore, if $p\equiv1\pmod{a}$, then we have $\T{aL+p\nu}(z)=\T{aL+\nu}(z)$, and hence this proves the proposition.
\end{proof}

The next proposition allows us to decompose the space $M_{k/2}(\Gamma_0(4N_La^2)\cap \Gamma_1(a), \chi_{4d_L})$ into the spaces $M_{k/2}(\Gamma_0(4N_La^2), \chi\chi_{4d_L})$.

\begin{proposition}\label{prop-modformspace-decomp}
	Let $k$, $M$, and $N$ be positive integers such that $M\mid N$, and let $\psi$ be a Dirichlet character modulo $N$. Then
	$$
	M_{k/2}(\Gamma_0(N)\cap \Gamma_1(M),\psi)= \bigoplus_{\chi} M_{k/2}(\Gamma_0(N),\chi\psi),
	$$
	where $\chi$ runs over all Dirichlet characters modulo $M$ such that $\chi(-1)=\psi(-1)$ if $k$ is odd, and $\chi(-1)=(-1)^{k/2}\psi(-1)$ if $k$ is even.
\end{proposition}
\begin{proof}
	The proposition for the case when $k$ is even was proved in \cite[Theorem 2.5]{Cho}. The case when $k$ is odd may also be proved in the same manner.
\end{proof}
Now let $k$ be an odd positive integer, $M$ and $N$ positive integers such that $M\mid N$, and $\psi$ a Dirichlet character modulo $N$.
For a prime number $p$, we define the Hecke operator $T(p^2)$ on the space $M_{k/2}(\Gamma_0(N)\cap \Gamma_1(M),\psi)$ by 
$$
T(p^2)=\bigoplus_{\chi} T|^N_{k/2,\chi\psi}(p^2),
$$
where $\chi$ runs over all Dirichlet characters modulo $M$, and $T|^N_{k/2,\chi\psi}(p^2)$ is the Hecke operator on the space $M_{k/2}(\Gamma_0(N),\chi\psi)$ defined in \cite{Shimura}.
The following theorem shows the relation between Hecke operators and Fourier coefficients of theta series $\T{aL+\nu}(z)$ of cosets.

\begin{theorem}\label{thm-Hecke-actonthetaofcoset}
	Let $k\ge 3$ be an odd integer, $aL+\nu$ a coset of rank $k$, $N_L$ the level of $L$, $d_L$ the discriminant of $L$, and let $p$ be a prime number.
	Put 
	$$
	(\T{aL+\nu}|T(p^2))(z)=\sum_{n\ge 0} b(n)q^n.
	$$ 
	If $p\mid 4N_La^2$, then $b(n)=r(p^2n,aL+\nu)$. If $p\nmid4N_La^2$, then
	$$
	b(n)=r(p^2n,aL+\nu)+\legendre{-1}{p}^{\lambda}\legendre{4d_Ln}{p}p^{\lambda-1} \cdot r(n,aL+\bar{p}\nu)+ \legendre{4d_L}{p^2}p^{k-2}\cdot r(n/p^2,aL+\bar{p}^2\nu),
	$$
	where $\lambda=(k-1)/2$, and $\bar{p}$ is an integer which is an inverse of $p$ modulo $a$.

\end{theorem}
\begin{proof}
	Noting that 
	$$
	\T{aL+\nu}(z) \in M_{k/2}(\Gamma_0(4N_La^2)\cap \Gamma_1(a), \chi_{4d_L})=\bigoplus_{\chi \pmod{a}}M_{k/2}(\Gamma_0(4N_La^2),\chi\chi_{4d_L}),
	$$
	let us write $\T{aL+\nu}(z)=\sum\limits_{\chi\pmod{a}} f_\chi (z)$ for some $f_\chi(z)=\sum\limits_{n\ge 0} a_\chi(n)q^n\in M_{k/2}(\Gamma_0(4N_La^2),\chi\chi_{4d_L})$.
	By the definition of $T(p^2)$ and by \cite[Theorem 1.7]{Shimura}, we have $b(n)=\sum_{\chi}b_{\chi}(n)$ with 
	\begin{equation}\label{eqn-coefficients-Heckeoperator-proof:1}
	b_{\chi}(n):=a_\chi(p^2n)+ \legendre{4d_L}{p}\legendre{-1}{p}^\lambda\legendre{n}{p}p^{\lambda-1}\chi(p)a_\chi(n)+\legendre{4d_L}{p^2}p^{k-2}\chi(p^2)a_\chi(n/p^2).
	\end{equation}
	Note that if $p\mid 4N_La^2$, then $\legendre{4d_L}{p}\chi(p)=0$, and hence $b(n)=\sum\limits_{\chi \pmod{a}} a_\chi(p^2n)=r(p^2n,aL+\nu)$. 
	Now we assume that $p\nmid 4N_La^2$. Let $m$ be an integer such that $(m,4N_La^2)=1$. Then there is a matrix $\gamma_m=\left(\begin{smallmatrix}
		\ast & \ast \\
		\ast & m
	\end{smallmatrix}\right)\in \Gamma_0(4N_La^2)$. Note that by \eqref{eqn-theta-transformed-by-Gamma0},
	$$
	\chi_{4d_L}(m)\T{aL+\bar{m}\nu}(z)=(\T{aL+\nu}|_{\frac{k}{2}}\gamma_m)(z)=\sum\limits_{\chi\pmod{a}} (f_\chi|_{\frac{k}{2}}\gamma_m)(z)= \chi_{4d_L}(m)\sum_{\chi\pmod{a}} \chi(m)f_\chi(z),
	$$
	where $\bar{m}$ is an integer which is an inverse of $m$ modulo $a$. 
	By comparing the Fourier coefficients of both sides, we have
	\begin{equation}\label{eqn-relation-on-coefficients-repnum-chi}
	r(n,aL+\bar{m}\nu)=\sum\limits_{\chi \pmod{a}} \chi(m)a_\chi(n).
	\end{equation}
	for any integer $n\ge 0$ and $(m,4N_La^2)=1$. Plugging in the equalities \eqref{eqn-relation-on-coefficients-repnum-chi} with $m=p$ and $m=p^2$ into \eqref{eqn-coefficients-Heckeoperator-proof:1}, we have the formula in the statement of the theorem.
\end{proof}

Now we define some notations for the ternary case, the case when $k=3$. We put
\[
U=\bigoplus_{\chi \pmod{a}} U_\chi \quad\text{and} \quad U^\perp=\bigoplus_{\chi \pmod{a}} U_\chi^\perp,
\]
where
\[
U_\chi = \bigoplus_{t : \text{square-free}} U_t(4N_La^2,\chi\chi_{4d_L}) \subseteq S_{3/2}(\Gamma_0(4N_La^2),\chi\chi_{4d_L}),
\]
and $U_\chi^\perp$ denotes the space orthogonal to $U_\chi$ in $S_{3/2}(\Gamma_0(4N_La^2),\chi\chi_{4d_L})$.
We note that each subspace occuring in the decomposition of $U_\chi$ is an eigenspace for the Hecke operators $T(p^2)$, as follows.
\begin{proposition}[Hilfssatz 2 of \cite{Schulze-PillotTernaryTheta}]\label{prop-U_t,chi-eigenspace}
	Let $p$ be a prime number such that $p\nmid 4N_La^2$. 
	Then $U_t(\Gamma_0(4N_La^2),\chi \chi_{d_L})$ is an eigenspace for $T(p^2)$ with eigenvalue $\chi(p)\legendre{-td_L}{p}(p+1)$.
\end{proposition}

\section{Some algebraic structure of lattice cosets}\label{section-algebraicstructureofcosets}
In this section, we introduce several lemmas regarding algebraic structures of lattice cosets, which will be used in the following sections.
\subsection{Genera of lattice cosets with the same conductor}
The following lemma shows some properties shared by the cosets of conductor $a$ in $(aL+\Z \nu) / (aL)$.
\begin{lemma}\label{lem-genus-structure}
	Let $s$ be an integer coprime to the conductor $a$ of $aL+\nu$. We have the following:
\begin{enumerate}[leftmargin=*, label={\rm(\arabic*)}]
	\item $O^+(aL+\nu)=O^+(aL+s\nu)$ and $O^+(aL_p+\nu)=O^+(aL_p+s\nu)$ for any prime $p$. 
	\item If $\pgen(aL+\nu)=\sqcup_{1\le i \le h_{\operatorname{gen}}} \pcls(aL_i+\nu_i)$, then $\pgen(aL+s\nu)=\sqcup_{1\le i \le h_{\operatorname{gen}}} \pcls(aL_i+s\nu_i)$.
	\item If $\pspn(aL+\nu)=\sqcup_{1\le i \le h_{\operatorname{spn}}} \pcls(aL_i+\nu_i)$, then $\pspn(aL+s\nu)=\sqcup_{1\le i \le h_{\operatorname{spn}}} \pcls(aL_i+s\nu_i)$.
\end{enumerate}
\end{lemma}
\begin{proof}
	(1) Let $\sigma\in O^+(aL+\nu)$. Then $\sigma(aL)=aL$ and $\sigma(\nu)-\nu\in aL$. Multiplying by $s$, we have  $\sigma(s\nu)-s\nu\in saL\subseteq aL$, hence $\sigma\in O^+(aL+s\nu)$.
	Likewise, we have $ O^+(aL+s^{k-1}\nu)\subseteq O^+(aL+s^k\nu)$ for any $k\in\N$. Note that $aL+\nu=aL+s^{\ord_a(s)}\nu$, where $\ord_a(s)\ge 1$ is the order of $s$ modulo $a$ in the multiplicative group $(\Z/a\Z)^\times$. Therefore, we have
	$$
	O^+(aL+\nu)\subseteq O^+(aL+s\nu)\subseteq O^+(aL+s^{\ord_a(s)}\nu) = O^+(aL+\nu),
	$$
	which proves the first statement. The equalities for local cosets follow in the same manner.
	
	(2) Noting that any coset in $\pgen(aL+s\nu)$ has conductor $a$, let $aK+\mu\in \pgen(aL+s\nu)$ be any coset in the proper genus of $aL+s\nu$.
	Then for any prime $p$, there exists $\sigma_p\in O^+(V_p)$ such that $\sigma_p(aK)=aL$ and $\sigma_p(\mu)-s\nu \in aL_p$.
	Let $\bar{s}$ be an integer which is an inverse of $s$ modulo $a$. Then $\sigma_p(\bar{s}\mu)-\bar{s}s\nu \in \bar{s}aL_p\subseteq aL_p$. Since $\bar{s}s\nu - \nu \in aL_p$, we have 
	$$
	\sigma_p(\bar{s}\mu)-\nu = \sigma_p(\bar{s}\mu)-\bar{s}s\nu + \bar{s}s\nu - \nu \in aL_p.
	$$
	Hence $\sigma_p (aK+\bar{s}\mu)=aL_p+\nu$ for any prime $p$, that is, $aK+\bar{s}\mu \in \pgen(aL+\nu)$.
	Therefore, $\sigma(aK+\bar{s}\mu)=aL_i+\nu_i$ for some $\sigma\in O^+(V)$ and $1\le i \le h_{\operatorname{gen}}$.
	One may easily show that this $\sigma\in O^+(V)$ satisfies $\sigma(aK+\mu)=aL_i+s\nu_i$.
	This proves the second statement.
	
	(3) The third statement may also be proved similarly as the proof of the second statement.
\end{proof}

\subsection{$p$-neighborhood of lattice cosets}\label{subsection-p-nbd-cosets}

Let $p$ be a prime number such that $p\nmid 4N_La^2$. Let $\bar{p}$ be an integer which is an inverse of $p$ modulo $a$. 
Define $R_p(aL+\nu)$ to be the set of cosets $aK+\mu$ with conductor $a$ satisfying the following:
\begin{enumerate}[label={\rm(\arabic*)}]
	\item $aK_q+\mu=aL_q+\bar{p}\nu$ for any prime $q\neq p$.
	\item $(L_p:L_p\cap K_p)=(K_p:L_p\cap K_p)=p$ and $Q(K_p)\Z_p=\Z_p$.
\end{enumerate}
From the second condition, $K_p$ is also a $\Z_p$-maximal lattice, hence $K_p$ is isometric to $L_p$ by an element in $O^+(V_p)$, due to the uniqueness of a $\Z_p$-maximal lattice up to isometry. Furthermore, by the local theory of lattices (cf. \cite[82:23]{OMBook}), there exists a basis $\{e_1,e_2,e_3\}$ of $L_p$ such that 
\begin{multline}\label{eqn:Kpbasis}
Q(e_1)=-d_L,\ Q(e_2)=Q(e_3)=0, B(e_1,e_2)=B(e_1,e_3)=0,\ B(e_2,e_3)=1,\\
 \text{ and } \{e_1,p^{-1}e_2,pe_3\} \text{ is a basis of } K_p.
\end{multline}
Noting that $aK_p+\mu=K_p+\mu=K_p$, $aL_p+\bar{p}\nu=L_p+\nu=L_p$, and $pK_p\subseteq L_p$, we have 
$$
aK+\mu\in \pgen(aL+\bar{p}\nu)\quad \text{and} \quad p(aK+\mu)\subseteq aL+\nu
$$ 
for any $aK+\mu\in R_p(aL+\nu)$ (for further details, see \eqref{eqn-neighborhood-inclusion} below). Hence for $n\in\N$, one may note that $r(p^2n,aL+\nu)>0$ if $r(n,aK+\mu)>0$ for some $aK+\mu \in R_p(aL+\nu)$.

For an $n\in \N$ and an $x\in R(p^2n,aL+\nu)$, we define
$$
\pi_p(x,aL+\nu) = |\{aK+\mu \in R_p(aL+\nu) : x\in pK_p\}|.
$$
These numbers for special types of lattices were considered in \cite{Ponomarev}, and were computed by means of quaternion orders.
The following lemma provide some properties about what we have just defined.
\begin{lemma}\label{lem-p-nbd-properties}
Let $aL+ \nu$ be a ternary coset with conductor $a$, $p$ a prime number such that $p\nmid 4N_La^2$, $x\in R(p^2n,aL+\nu)$ and $d_L$ the discriminant of $L$. 
Under the notations given above, we have the following.
\begin{enumerate}[leftmargin=*, label={\rm(\arabic*)}]
	\item For any $k\in\Z_{\ge 0}$, $aK+\mu \in R_p(aL+\nu)$ if and only if $aK+\bar{p}^{k}\mu \in R_p(aL+\bar{p}^k\nu)$.
	\item $|R_p(aL+\nu)|=p+1$.
	\item $\pi_p(x,aL+\nu)=\begin{cases}
		1 & \text{if } x \in L_p\setminus pL_p,\\
		1+\legendre{-d_Ln}{p}& \text{if } x \in pL_p\setminus p^2L_p,\\
		p+1 & \text{if } x\in p^2L_p.
	\end{cases}$
\end{enumerate}
\end{lemma}
\begin{proof}
(1) Let $k\in\Z_{\ge 0}$. For any prime $q\neq p$, note that $aK_q+\bar{p}^k\mu=aL_q+\bar{p}^{k+1}\nu$ if and only if
$$
K_q = L_q \text{ and } \bar{p}^k\mu-\bar{p}^{k+1}\nu \in aL_q
$$
If $(a,q)=1$, then $aK_q+\bar{p}^k\mu = K_q + \bar{p}^k\mu = K_q = L_q = aL_q+\bar{p}^{k+1}\nu$.
Otherwise, we have $q\mid a$ so that $(\bar{p},q)=1$. Hence $\mu- \bar{p}\nu \in aL_q$ if and only if $\bar{p}^k\mu- \bar{p}^{k+1}\nu \in aL_q$.
Therefore, from the definition of the set $R_p(aL+\bar{p}^{k}\nu)$, $aK+\mu \in R_p(aL+\nu)$ if and only if $aK+\bar{p}^{k}\mu \in R_p(aL+\bar{p}^{k+1}\nu)$.
\vspace{.05in}

\noindent
(2) Note that two cosets are equal if and only if they are locally equal at all prime spots. Hence by the definition of the set $R_p(aL+\nu)$, we need only to investigate how many distinct $\Z_p$-lattices $K_p$ satisfy $(L_p:L_p\cap K_p)=(K_p:L_p\cap K_p)=p$ and $Q(K_p)\Z_p=\Z_p$. Putting $M=pK_p$ and recalling \eqref{eqn:Kpbasis}, this is equivalent to finding all sublattices $M$ of $L_p$ with elementary divisors (also called invariant factors) $1,p,p^2$ such that $Q(M)\Z_p=p^2\Z_p$.

To find possible sublattices $M$ of $L_p$, we fix a basis $\{e_1,e_2,e_3\}$ of $L_p$ such that 
$$
Q(e_1)=-d_L,\ Q(e_2)=Q(e_3)=0,\ B(e_1,e_2)=B(e_1,e_3)=0,\ \text{and } B(e_2,e_3)=1, 
$$
which is known to exist by the local theory of lattices. 

We then choose three $\Z_p$-linearly-independent elements 
\[
a_1=a_{11}e_1+a_{21}e_2+a_{31}e_3,\ a_2=a_{12}e_1+a_{22}e_2+a_{32}e_3,\ a_3=a_{13}e_1+a_{23}e_2+a_{33}e_3
\]
of $L_p$ that generate sublattices with the desired properties. Note that if we correspond any sublattice $\Z_pa_1+\Z_pa_2+\Z_pa_3$ of $L_p$ to the matrix $A=(a_{ij})\in M_3(\Z_p)$ with $\det(A)\neq 0$, then the sublattices of $L_p$ with invariant factors $1,p,p^2$ correspond to the left cosets of the double cosets
$$
GL_3(\Z_p)\cdot \diag(1,p,p^2)\cdot GL_3(\Z_p)/GL_3(\Z_p).
$$
Moreover, every left coset from the above contains exactly one element in the set 
$$
\begin{aligned}
\mathcal{C}=\{C=(c_{ij})\in M_3(\Z) : c_{ii}=p^{k_i}\text{ for some }&k_i\in\Z_{\ge 0},\ c_{11}c_{22}c_{33}=p^3,\\ 
&c_{ij}=0 \text{ if }i<j, \text{ and } 0\le c_{ij}< c_{ii} \text{ if } j\le i \}
\end{aligned}
$$
of lower-triangular matrices of determinant $p^3$. Therefore, if we search for the matrices in $\mathcal{C}$ whose corresponding sublattice has norm $p^2\Z_p$, and check whether these indeed have invariant factors $1,p,p^2$ in $L_p$, one may conclude that our $pK_p$ is one of the following:
\begin{equation}\label{sublattices-of-Lp}
\begin{aligned}
	&M_{1,u,v}=\Z_p(e_1+ue_2+ve_3)+\Z_p(pe_2+pwe_3)+\Z_p(p^2e_3),\\
	&M_2=\Z_p pe_1+\Z_p e_2 + \Z_p p^2e_3, \quad \text{ and } \quad M_3=\Z_p pe_1+\Z_p p^2e_2 + \Z_p e_3,
\end{aligned}
\end{equation}
where $0< u,w <p$ and $0<v<p^2$ are integers satisfying 
\begin{equation}\label{condition-u,v,w}
-d_L+2uv\equiv 0\pmod{p^2} \quad \text{ and } \quad v+uw\equiv 0\pmod{p}.
\end{equation}
Note that $u,v,w$ are not divisible by $p$, and that if $0<u<p$ is determined, there is only one choice for $0<v<p^2$, and hence $0<w<p$ is also determined. Therefore, there are $p-1+2=p+1$ cosets $aK+\mu$ in $R_p(aL+\nu)$, namely, the cosets $aK+\mu\in R_p(aL+\nu)$ such that $pK_p$ is one of the $\Z_p$ sublattices $M_{1,u,v}$, $M_2$, or $M_3$ of $L_p$.
\vspace{.05in}

\noindent
(3) From the definition of $\pi_p(x,aL+\nu)$, it suffices to count the number of sublattices of $L_p$ in \eqref{sublattices-of-Lp} containing a given $x \in R(p^2n,aL+\nu)$.
Since any sublattice in \eqref{sublattices-of-Lp} contains $p^2L_p$, we have $\pi_p(x,aL+\nu)=p+1$ if $x\in p^2L_p$.
On the other hand, note that for a $\sigma_p\in O(L_p)$, we have
$$
\begin{aligned}
	(L_p:L_p\cap \sigma_p(K_p))&=(\sigma_p(L_p):\sigma_p(L_p)\cap\sigma_p(K_p))=(L_p:L_p\cap K_p),\\
	(\sigma_p(K_p):L_p\cap \sigma_p(K_p))&=(\sigma_p(K_p):\sigma_p(L_p)\cap\sigma_p(K_p))=(K_p:L_p\cap K_p).
\end{aligned}
$$
Hence, if there is a vector $x'\in L_p$ with $Q(x)=Q(x')$ and a $\sigma_p\in O(L_p)$ (not necessarily a rotation) such that $x'=\sigma_px$, then
$$
\begin{aligned}
	\pi_p(x,aL+\nu)&=|\{aK+\mu \in R_p(aL+\nu) : x\in pK_p\}|\\
	&=|\{aK+\mu \in R_p(aL+\nu) : \sigma_p x\in p\cdot\sigma_p(K_p)\}|\\
  	&=|\{aK'+\mu' \in R_p(aL+\nu) : x'\in pK_p')\}|,
\end{aligned}
$$
where $aK'+\mu'$ is the coset in $R_p(aL+\nu)$ such that $aK_p'+\mu'=K'_p=\sigma_p(K_p)$ and $aK_q'+\mu'=aK_q+\mu$ for any prime $q\neq p$.

If $x \in L_p\setminus pL_p$, then $x':=e_2+\frac{p^2n}{2}e_3\in L_p\setminus pL_p$ satisfies $Q(x)=Q(x')$. Hence by \cite[Theorem 5.4.1]{KiBook}, there exists a $\sigma_p\in O(L_p)$ such that $\sigma_p x = x'$.
Moreover, among the sublattices of $L_p$ in \eqref{sublattices-of-Lp}, only $M_2$ contains $x'$. Hence, $\pi_p(x,aL+\nu)=1$.

Now we consider the case when $x\in pL_p\setminus p^2L_p$.
First, assume that $\legendre{-d_Ln}{p}=1$. Then $n=-\varepsilon^2d_L$ for some $\varepsilon\in\Z_p^\times$.
Hence $x'=p\varepsilon e_1$ satisfies $Q(x')=-p^2\varepsilon^2d_L=p^2n=Q(x)$ and $x/p,x'/p\in L_p\setminus pL_p$.
Again by \cite[Theorem 5.4.1]{KiBook}, there exists  a $\sigma_p\in O(L_p)$ such that $\sigma_p (x/p) = x'/p$, hence $\sigma_p x = x'$.
It is clear that both $M_2$ and $M_3$ contain $x'$.
Assume that a sublattice $M_{1,u,v}$ contains $x'=p\varepsilon e_1$. Then
$$
p\varepsilon e_1 = p\varepsilon(e_1+ue_2+ve_3) + b(pe_2+pwe_3)+c(p^2 e_3)
$$
for some $b,c\in\Z_p$, so that $b=-\varepsilon u$ and $\varepsilon(v-wu)=-cp$. Hence $v\equiv uw \pmod{p}$. However, by \eqref{condition-u,v,w}, we have $2v\equiv v+uw\equiv 0\pmod{p}$, which is a contradiction. Hence 
\[
\pi_p(x,aL+\nu)=2=1+\legendre{-d_Ln}{p}.
\]
In the case when $\legendre{-d_Ln}{p}=-1$ or $0$, one may argue in the same way to show that $\pi_p(x,aL+\nu)=0$ or $1$, respectively, by taking $x'=pe_2+\frac{pn}{2}e_3$; the details are left to the interested reader. This completes the proof of the lemma.
\end{proof}

\section{Hecke Operators on the theta series}\label{section-Hecke-on-theta}

In this section, we discuss how the action of the Hecke operators on the theta series of cosets is related to its $p$-neighborhood ($p\nmid 4N_La^2$).  
For two cosets $aL+\nu$ and $aM+\xi$, we put 
$$
c_p(aL+\nu,aM+\xi)= |\{aK+\mu\in R_p(aL+\nu) : aK+\mu \in \pcls(aM+\xi)  \}|.
$$
Let $aK+\mu\in R_p(aL+\nu)$ and let $\{aL_i+\nu_i\}_{1\le i \le h}$ be a set of representatives of proper classes of $\pgen(aL+\nu)$.
Since $R_p(aL+\nu)\subseteq \pgen(aL+\bar{p}\nu)$, Lemma \ref{lem-genus-structure} implies that $aK+\mu \in \pcls(aL_j+\bar{p}\nu_j)$ for some $1\le j \le h$.
Moreover, if $\sigma(aK+\mu)=aL_j+\bar{p}\nu_j$ for some $\sigma \in O^+(V)$, then $\sigma(aK+\bar{p}^k\mu)=aL_j+\bar{p}^{k+1}\nu_j$ for any $k\in\Z_{\ge 0}$.
Therefore, together with Lemma \ref{lem-p-nbd-properties} (1), we have 
$$
c_p(aL+\nu, aL_j+\bar{p}\nu_j) = c_p(aL+\bar{p}^{k}\nu, aL_j+\bar{p}^{k+1}\nu_j)
$$
for any $k\in\Z_{\ge 0}$. 
Hence for $1\le i,j \le h$, the following are defined independent of $k\in\Z_{\ge 0}$:
$$
\pi_{ij}(p^2):=c_p(aL_i+\bar{p}^k\nu_i, aL_j+\bar{p}^{k+1}\nu_j)=c_p(aL_i+\nu_i, aL_j+\bar{p}\nu_j).
$$
 
For any $k\in\Z_{\ge 0}$, put
$$
\T{G_k}(z):=\begin{bmatrix}
	\T{aL_1+\bar{p}^k\nu_1}(z)\\
	\vdots\\
	\T{aL_h+\bar{p}^k\nu_h}(z)
\end{bmatrix}.
$$
We now describe a generalization of the Eichler's commutation relation for cosets as follows.
\begin{theorem}\label{thm-Hecke-nbd}
	Let $aL+\nu$ be a ternary coset with conductor $a$, and let $p$ be a prime number such that $p\nmid 4N_La^2$. We have
	\begin{equation}\label{eqn-Hecke-nbd}
	(\T{aL+\nu}|T(p^2))(z)=\sum\limits_{aK+\mu \in R_p(aL+\nu)} \T{aK+\mu}(z).		
	\end{equation}
	Furthermore, for any $k\in\Z_{\ge 0}$, we have
	$$
		\T{G_k} | T(p^2) = (\pi_{ij}(p^2))\T{G_{k+1}}.
	$$
	
\end{theorem}
\begin{proof}
	According to the discussion above, the furthermore part of the theorem follows immediately once we prove \eqref{eqn-Hecke-nbd}.
	Hence it suffices to show \eqref{eqn-Hecke-nbd}, that is, by Theorem \ref{thm-Hecke-actonthetaofcoset}, for any $n\in\N$,
	\begin{equation}\label{eqn-Hecke-nbd-coeff}
	r(p^2n, aL+\nu) + \legendre{-4d_Ln}{p} r(n,aL+\bar{p}\nu) + p\cdot r(n/p^2,aL+\bar{p}^2\nu) = \sum\limits_{aK+\mu \in R_p(aL+\nu)} r(n,aK+\mu).
	\end{equation}
	
	To show \eqref{eqn-Hecke-nbd-coeff}, we will count the sum $\sum\limits_{x\in R(p^2n, aL+\nu)} \pi_p (x,aL+\nu)$ in two different ways.
	First, note that for any $aK+\mu \in R_p(aL+\nu)$, we have 
	\begin{equation}\label{eqn-neighborhood-inclusion}
	\begin{cases}
	p(aK+\mu)_q=p(aL+\bar{p}\nu)_q=aL_q+\bar{p}p\nu=aL_q+\nu \text{ for any prime } q\neq p,	\\
	p(aK+\mu)_p=pK_p \subseteq L_p=(aL+\nu)_p.
	\end{cases}
	\end{equation}
	Hence, $x\in R(p^2n, aL+\nu)$ with $x\in pK_p$ if and only if $x\in p(aK+\mu)$ with $Q(x)=p^2n$. Thus,
	\begin{equation}\label{eqn-Hecke-nbd-proof:1}
	\begin{aligned}
		\sum\limits_{x\in R(p^2n, aL+\nu)} \pi_p (x,aL+\nu) &= \sum\limits_{x\in R(p^2n, aL+\nu)}\sum\limits_{aK+\mu \in R_p(aL+\nu) \atop x\in pK_p} 1&	\\
		&= \sum\limits_{aK+\mu \in R_p(aL+\nu)}\sum\limits_{ x \in p(aK+\mu) \atop Q(x)=p^2n} 1 &= 	\sum\limits_{aK+\mu \in R_p(aL+\nu)} r(n,aK+\mu),
	\end{aligned}
	\end{equation}
	and the last term is equal to the the right-hand side of \eqref{eqn-Hecke-nbd-coeff}.

	On the other hand, by Lemma \ref{lem-p-nbd-properties} (3), the sum $\sum\limits_{x\in R(p^2n, aL+\nu)} \pi_p (x,aL+\nu)$ is equal to 
	\begin{equation}\label{eqn-Hecke-nbd-proof:2}
	\begin{aligned}
		 &\sum\limits_{x\in L_p\setminus pL_p} \pi_p (x,aL+\nu)+\sum\limits_{x\in pL_p\setminus p^2L_p} \pi_p (x,aL+\nu) +\sum\limits_{x\in p^2L_p} \pi_p (x,aL+\nu)\\
		=&\sum\limits_{x\in L_p\setminus pL_p} 1 + \sum\limits_{x\in pL_p\setminus p^2L_p} \left(1+\legendre{-d_Ln}{p}\right) + \sum\limits_{x\in p^2L_p} (1+p)\\
		=&\left[\sum\limits_{x\in R(p^2n, aL+\nu)} 1 \right] + \legendre{-4d_Ln}{p}\cdot\left[\sum\limits_{x\in pL_p\setminus p^2L_p} 1 \right]+ p\cdot \left[\sum\limits_{x\in p^2L_p} 1\right],
	\end{aligned}
	\end{equation}
where we omit the condition $Q(x)=p^2n$ in the intermediary sums for ease of notation. Note that if $p\mid n$, then $\legendre{-4d_Ln}{p}=0$, and if $p\nmid n$, then $x/p \in L_p$ if and only if $x/p \in L_p\setminus pL_p$ since $Q(x/p)=n$ and $Q(L_p)\subseteq\Z_p$.
	Moreover, $x\in aL+\nu$ with $x\in p^k L_p$ if and only if $x/p^k \in aL+\bar{p}^k \nu$ for any $k\in \Z_{\ge 0}$.
	Hence the last equation of \eqref{eqn-Hecke-nbd-proof:2} is equal to
	$$
		r(p^2n, aL+\nu) + \legendre{-4d_Ln}{p} r(n,aL+\bar{p}\nu) + pr(n/p^2,aL+\bar{p}^2\nu),
	$$
	which is equal to the the left-hand side of \eqref{eqn-Hecke-nbd-coeff}. 
	This completes the proof of the theorem.
\end{proof}

We next use the above theorem to investigate the first piece in the splitting \eqref{eqn-splitting-of-thetaofcosets}, the Eisenstein series part of the theta series of a coset $aL+\nu$. Specifically, we evaluate the action of the Hecke operators $T(p^2)$ on the theta series $\T{\pgen(aL+\nu)}(z)$ for the genus. 

\begin{theorem}	\label{thm-Hecke-genera}
	Let $aL+\nu$ be a ternary coset with conductor $a$, and let $p$ be a prime number such that $p\nmid 4N_La^2$. Then, we have
	\begin{gather*}
	(\T{\pgen(aL+\nu)}|T(p^2))(z)=(p+1)\T{\pgen(aL+\bar{p}\nu)}(z).
	\end{gather*}
\end{theorem}
\begin{proof}
	Let $\{aL_i+\nu_i\}_{1\le i \le h}$ be a set of representatives of proper classes of $\pgen(aL+\nu)$.
	Note that $\sum\limits_{i=1}^h \frac{1}{o^+(aL_i+\nu_i)}=	\sum\limits_{j=1}^h \frac{1}{o^+(aL_j+\bar{p}\nu_j)}$ by Lemma \ref{lem-genus-structure} (1).
	Thus, by Theorem \ref{thm-Hecke-nbd}, it suffices to show that
	\begin{equation}\label{eqn-Hecke-pgen-proof:1}
	\sum\limits_{i=1}^h \frac{1}{o^+(aL_i+\nu_i)}\sum\limits_{j=1}^h c_p(aL_i+\nu_i, aL_j+\bar{p}\nu_j) \T{aL_j+\bar{p}\nu_j}(z) = (p+1) \sum\limits_{j=1}^h \frac{\T{aL_j+\bar{p}\nu_j}(z)}{o^+(aL_j+\bar{p}\nu_j)}.
	\end{equation}
	Note that for $aK+\mu\in R_p(aL_i+\nu_i)$, if $aK+\mu\in\pcls(aL_j+\bar{p}\nu_j)$, then there is a $\sigma\in O^+(V)$ such that $aK+\mu = \sigma(aL_j+\bar{p}\nu_j)$.
	Hence we have
	$$
	o^+(aL_j+\bar{p}\nu_j)c_p(aL_i+\nu_i,aL_j+\bar{p}\nu_j) = |\{\sigma \in O^+(V) : \sigma(aL_j+\bar{p}\nu_j)\in R_p(aL_i+\nu_i)\}|,
	$$
	and it follows from the definition of $R_p(aL_i+\nu_i)$ that $\sigma(aL_j+\bar{p}\nu_j)\in R_p(aL_i+\nu_i)$ if and only if $\sigma^{-1}(aL_i+\bar{p}^2\nu_i)\in R_p(aL_j+\bar{p}\nu_j)$.
	Hence we have
	$$
	\begin{aligned}
	o^+(aL_j+\bar{p}\nu_j)c_p(aL_i+\nu_i,aL_j+\bar{p}\nu_j)&=|\{\sigma \in O^+(V) : \sigma^{-1}(aL_i+\bar{p}^2\nu_i)\in R_p(aL_j+\bar{p}\nu_j)\}|\\		
	&=o^+(aL_i+\bar{p}^2\nu_i)c_p(aL_j+\bar{p}\nu_j,aL_i+\bar{p}^2\nu_i)\\
	&=o^+(aL_i+\nu_i)c_p(aL_j+\bar{p}\nu_j,aL_i+\bar{p}^2\nu_i),
	\end{aligned}
	$$
	where the last equality holds by Lemma \ref{lem-genus-structure} (1). Hence, the left-hand side of \eqref{eqn-Hecke-pgen-proof:1} is equal to
	$$
	\begin{aligned}
		\sum\limits_{j=1}^h \sum\limits_{i=1}^h \frac{c_p(aL_i+\nu_i, aL_j+\bar{p}\nu_j)}{o^+(aL_i+\nu_i)}  \T{aL_j+\bar{p}\nu_j}(z)&=\sum\limits_{j=1}^h \sum\limits_{i=1}^h \frac{c_p(aL_j+\bar{p}\nu_i, aL_i+\bar{p}^2\nu_i)}{o^+(aL_j+\bar{p}\nu_j)}  \T{aL_j+\bar{p}\nu_j}(z)\\
		&=\sum\limits_{j=1}^h \frac{\T{aL_j+\bar{p}\nu_j}(z) }{o^+(aL_j+\bar{p}\nu_j)}  \sum\limits_{i=1}^h c_p(aL_j+\bar{p}\nu_j, aL_i+\bar{p}^2\nu_i)\\
		&=(p+1)\sum\limits_{j=1}^h \frac{\T{aL_j+\bar{p}\nu_j}(z)}{o^+(aL_j+\bar{p}\nu_j)},
	\end{aligned}
	$$
where in the last step we note that the inner sum evaluates to $|R_p(aL_j+\bar{p}\nu_j)|$, which is $p+1$ by Lemma \ref{lem-p-nbd-properties} (2). This proves \eqref{eqn-Hecke-pgen-proof:1}, completing the proof of the theorem.
\end{proof}

In the special case that $p\equiv \pm 1\pmod{a}$ (and $p\nmid 4N_L$), Theorem \ref{thm-Hecke-genera} yields that $\T{\pgen(aL+\nu)}(z)$ is an eigenform of the Hecke operators $T(p^2$) with eigenvalue $p+1$, yielding the conclusion that the theta series for the genus is an Eisenstein series.

%
%
%

\begin{corollary}\label{cor-theta-pgen-gen-equal}
	The theta series $\T{\pgen(aL+\nu)}(z)$ is an Eisenstein series. In particular,
	$$
	\T{\pgen(aL+\nu)}(z)=\T{\gen(aL+\nu)}(z).
	$$
\end{corollary}
\begin{proof}
	Let $\{aL_i+\nu_i\}_{1\le i \le h}$ be a set of representatives of proper classes of $\pgen(aL+\nu)$.
	According to the result of Shimura \cite{ShimuraCongruence}, $\T{aL_i+\nu_i}(z) - \T{\gen(aL+\nu)}(z)$ is a cusp form for any $1\le i \le h$, hence
	$$
	\T{\pgen(aL+\nu)}(z)-\T{\gen(aL+\nu)}(z) = \frac{1}{\sum_{i=1}^h o^+(aL_i+\nu_i)^{-1}}\sum\limits_{i=1}^h \frac{\T{aL_i+\nu_i}(z) - \T{\gen(aL+\nu)}(z)}{o^+(aL_i+\nu_i)}
	$$
	is also a cusp form. Since Shimura \cite{ShimuraCongruence} also showed that $\T{\gen(aL+\nu)}(z)$ is an Eisenstein series, it suffices to show that $\T{\pgen(aL+\nu)}(z)$ is an Eisenstein series.
	By Proposition \ref{prop-modformspace-decomp}, it is enough to show that the projection $\pi'(\T{\pgen(aL+\nu)}(z))$ to $U^\perp$ and the projections $\pi_{t,\chi}(\T{\pgen(aL+\nu)}(z))$ to $U_{t}(4N_La^2,\chi\chi_{4d_L})$ for any positive square-free integer $t$ and for any even Dirichlet character $\chi$ modulo $a$ are equal to zero.
	
	Let $p$ be a prime such that $p\equiv 1 \pmod{a}$ and $p\nmid 4N_La^2$. Since the projection $\pi'$ commutes with the Hecke operator $T(p^2)$ and $\bar{p}\equiv 1\pmod{a}$, Theorem \ref{thm-Hecke-genera} implies that
	$$
	\begin{aligned}
	\pi'(\T{\pgen(aL+\nu)}(z))|T(p^2) &= \pi'(\T{\pgen(aL+\nu)}|T(p^2)(z))\\
	& = \pi'( (p+1)\T{\pgen(aL+\nu)}(z)) = (p+1)\pi'(\T{\pgen(aL+\nu)}(z)).	
	\end{aligned}
	$$
	If $\pi'(\T{\pgen(aL+\nu)}(z))\neq 0$, then its Shimura lifts would be non-zero cusp forms of weight $2$, which would be eigenfunctions of $T(p)$ for all $p\nmid 4N_La^2$ such that $p\equiv 1\pmod{a}$ with eigenvalue $p+1$. This contradicts the Weil bounds proven by Deligne \cite{Deligne}, and hence $\pi'(\T{\pgen(aL+\nu)}(z))=0$.
	
	By Theorem \ref{thm-Hecke-genera} and Proposition \ref{prop-U_t,chi-eigenspace}, and since the Hecke operator commutes with $\pi_{t,\chi}$, we have for any prime number $p$ such that $p\nmid 4N_La^2$,
	\begin{equation}\label{eqn-theta-pgen-gen-equal-proof:1}
	\begin{aligned}
		(p+1)\pi_{t,\chi}(\T{\pgen(aL+\bar{p}\nu)}(z)) &= \pi_{t,\chi}(\T{\pgen(aL+\nu)}|T(p^2)(z))\\
		&=(\pi_{t,\chi}(\T{\pgen(aL+\nu)}))|T(p^2)(z)\\
		&=\chi(p)\legendre{-td_L}{p}(p+1)\pi_{t,\chi}(\T{\pgen(aL+\nu)}(z)).
	\end{aligned}
	\end{equation}
	Note that if there is a prime $p$ satisfying 
	\begin{equation}\label{condition-prime-pequiv1moda}
		p \nmid 4N_La^2, \ p\equiv 1 \pmod{a} \  \text{and} \ \legendre{-td_L}{p}=-1,
	\end{equation}
	then $\chi(p)=1$ and hence \eqref{eqn-theta-pgen-gen-equal-proof:1} implies that $\pi_{t,\chi}(\T{\pgen(aL+\nu)}(z))=0$.
	
	Let $s$ be the square-free part of $td_L$, $s_0$ be the odd part of $s$, $g=(a,s_0)$, and $s_0=g\cdot g'$.
	Note that $(a,g')=1$ since $s_0$ is square-free.
	Assume that $g'\neq 1$. Then for any $p\equiv 1 \mod{a}$, since $p\equiv 1\pmod {g}$, quadratic reciprocity implies that
	$$
	\legendre{-td_L}{p}=\legendre{-s/s_0}{p}\legendre{g}{p}\legendre{g'}{p}=\legendre{-s/s_0}{p} (-1)^{\frac{g-1}{2}\cdot \frac{p-1}{2}}\cdot 1 \cdot (-1)^{\frac{g'-1}{2}\cdot \frac{p-1}{2}} \legendre{p}{g'}.
	$$
By the Chinese remainder theorem, we may choose $p\equiv 1\pmod{8}$ so that the above simplifies as 
\[
\legendre{-td_L}{p}= \legendre{p}{g'}.
\]
By the Chinese remainder theorem and Dirichlet's theorem on primes in arithmetic progressions, we may choose $p$ in any congruence class relatively prime to $g'$, and hence we may choose $\legendre{p}{g'}=-1$, yielding a prime satisfying the conditions in \eqref{condition-prime-pequiv1moda}, and hence we are done.

	Now we may assume that $g'=1$, or equivalently, $s\mid 2a$. We first consider the case when $\ord_2(a)\le 2$ and $2\mid s$.
	We may take a prime $p\equiv 1 \pmod{a}$ such that $p\equiv 5 \pmod{8}$. Then, using quadratic reciprocity and noting that $p\equiv 1\pmod{s_0}$ because $s_0\mid a$, we have
	$$
	\legendre{-td_L}{p}=\legendre{-2}{p}\legendre{s_0}{p}=(-1)\cdot\legendre{p}{s_0} =-1.
	$$
	Also, when $\ord_2(a)\le 1$ and $s\equiv 1\pmod{4}$, one may similarly show that any prime $p\equiv 1\pmod{a}$ with $p\equiv 3\pmod{4}$ satisfies $\legendre{-td_L}{p}=-1$.
	Since there exists $p$ satisfying the conditions in \eqref{condition-prime-pequiv1moda} in either case, we are done with these cases. 

	Now we are left with the cases when $s\mid a$ and either $8\mid a$, $\ord_2(a)=2$ with $2\nmid s$, or $\ord_2(a)\le 1$ with $s\equiv 3\pmod{4}$.
	Note that in these cases, if $p\equiv -1 \pmod{a}$, then one may check that
	$$
	\legendre{-td_L}{p}=\legendre{-s/s_0}{p}(-1)^{\frac{s_0-1}{2}\cdot \frac{p-1}{2}}\legendre{p}{s_0} = \legendre{-s/s_0}{p}(-1)^{\frac{s_0-1}{2}\cdot \frac{p-1}{2}}(-1)^{\frac{s_0-1}{2}}=-1.
	$$
	Furthermore, for any prime $p\equiv -1 \pmod{a}$,  we have by Lemma \ref{lem-genus-structure} that
	$$
	\T{\pgen(aL+\bar{p}\nu)}(z)=\T{\pgen(aL-\nu)}(z)=\T{\pgen(aL+\nu)}(z).
	$$
	Therefore, with a prime $p\equiv -1\pmod{a}$, noting that $\chi(p)=1$ for any even character $\chi$ modulo $a$, \eqref{eqn-theta-pgen-gen-equal-proof:1} again implies that $\pi_{t,\chi}(\T{\pgen(aL+\nu)}(z))=0$.
	This completes the proof of the corollary.
\end{proof}

\section{The theta series of the spinor genera}\label{section-unarythetafunctionpart}
In this section, we use measure theory to obtain Theorems \ref{thm-repnumofspinor}, \ref{thm-diff-thetaofpspn-thetaofpgen} and Corollary \ref{cor-determining-coefficientsofunarytheta} on relations of the representation numbers
$r(n,\pspn(aM+\xi))$ for proper spinor genera in the same proper genus.
Actually, Teterin \cite{Teterin} already stated Theorem \ref{thm-repnumofspinor} and the first part of Theorem \ref{thm-diff-thetaofpspn-thetaofpgen}, and proved those by giving a brief explanation.
Moreover, he claimed a stronger statement in \cite[Theorem 1 (2)]{Teterin} on an explicit formula for the difference of the representation numbers for two proper spinor genera. However, there seems to be a minor error in his proof which leads to an incorrect statement (see Remark \ref{rem:Teterin-counterexample} for a counter-example). Although we believe that his assertion can be modified to yield a correct statement, we propose an alternative way to obtain such an explicit formula in Corollary \ref{cor-determining-coefficientsofunarytheta}.
For the rest of this section, we provide some detailed explanation for the proof of the theorems for the convenience of the reader. 
The idea of using measure theory originally comes from Kneser \cite{Kneser} and Schulze-Pillot \cite{Schulze-PillotMassTernarySpinor}.

Let $V$ be a quadratic space and $x\neq0$ be a non-zero vector of $V$. Let $O^+(V,x)$ denote the fixed group of $x$ in $O^+(V)$, and let $O^+(aL+\nu,x)=O^+(V,x)\cap O^+(aL+\nu)$.

A representation $(x,aL+\nu)$ of a number $n$ by a coset $aL+\nu$ is given by a $x\in aL+\nu$ with $Q(x)=n$.
We say that two representations $(x,aL+\nu)$ and $(y,aM+\xi)$ {\em are equivalent} or {\em belong to the same representation class}  if there is a $u\in O^+(V)$ with $ux=y$ and $u(aL+\nu)=aM+\xi$, in which case we write $(x,aL+\nu)\cong (y,aM+\xi)$.
In particular, we have 
$$
	\begin{array}{ll}
		\text{$(x,aL+\nu)\cong (y,aL+\nu)$} &\text{if $ux=y$ with $u\in O^+(aL+\nu)$,}\\
		\text{$(x,aL+\nu)\cong (x,aM+\xi)$} &\text{if $u(aL+\nu)=aM+\xi$ with $u\in O^+(V,x)$.}
	\end{array}
$$
The class of a representation $(x,aL+\nu)$ is denoted by $[(x,aL+\nu)]$. Local representation classes are defined in the same way. We abuse notation and write $\cong$ for local equivalence as well.

For $x\in V$ and $y\in aM+\xi$ with $Q(x)=Q(y)$, it follows from Witt's theorem that there is a representation $(x,aM'+\xi')$ that is equivalent to $(y,aM+\xi)$.
Hence, if we are only interested in the classes of represention of a number $n$, we can restrict ourselves to representations with fixed $x\in V$ satisfying $Q(x)=n$.

We say $(x,aL+\mu)$ and $(y,aM+\xi)$ {\em belong to the same genus} if $(x,aL_p+\nu)\cong (y,aM_p+\xi)$ for every prime spot $p$ including $\infty$.
Note that the classes of representations of $Q(x)$ by cosets in the genus of $aL+\nu$ are in one-to-one correspondence with the double cosets $O^+(V,x)uO_A^+(aL+\nu)$ with $u\in O_A^+(V)$ and $x \in u(aL+\nu)$, and for $u\in O_A^+(V)$ for which $x \in u(aL+\nu)$, the genus of $(x,u(aL+\nu))$ is given by the double coset $O_A^+(V,x)uO_A^+(aL+\nu)$.

Now we consider two Haar measures
$$
\mu=\mu_\infty \times \prod\limits_{p <\infty} \mu_p \quad \text{and} \quad \lambda=\lambda_\infty \times \prod\limits_{p <\infty} \lambda_p
$$ 
on $O_A^+(V,x)=O_\infty^+(V,x) \times \prod\limits_{p<\infty} O_p^+(V_p,x)$ and $O_A^+(V)=O_\infty^+(V) \times \prod\limits_{p<\infty} O_p^+(V_p)$, respectively. Since we are dealing with the case when $V$ is positive definite, the measures are finite.
The {\em measure $r(x,aL+\nu)$ of the representation} $(x,aL+\nu)$ is defined as 
$$
r(x,aL+\nu)=\mu_\infty(O_\infty^+(V,x)/O^+(aL+\nu,x))=\frac{\mu_\infty(O_\infty^+(V,x))}{o^+(aL+\nu,x)}.
$$
This value is uniquely determined once the normalization of $\mu$ is determined. Since we are only interested in comparing ratios of measures with each other, the normalization factor always cancels and hence does not matter for our consideration (see Lemma \ref{lem-measure-repnum}).
The only property we need is that the normalization can be carried out in such a way that under $u\in O^+(V,x)$ the measure on $O_\infty^+(V,x)$ transfers into the measure on $O_\infty^+(V,ux)$.
Hence, the measure is invariant for equivalent representations and is also referred to as the representation measure of the representation class.

Note that the system of representatives of classes in the genus of a representation $(x,aL+\nu)$ may be obtained from $O_A^+(V,x)\cdot (x,aL+\nu)$ and the classes of representations in the genus of $(x,aL+\nu)$ intersected with $O_A^+(V,x)\cdot (x,aL+\nu)$ are in one-to-one correspondence with the double cosets $O^+(V,x)uO_A^+(aL+\nu,x)$ with $u\in O_A^+(V,x)$.

Now we provide three lemmas which translate the language of the number of representations $r(n,aL+\nu)$ of $n\in\N$ into that of the measure of a representation $(x,aL+\nu)$ with $Q(x)=n$.

\begin{lemma}\label{lem-measure-class}
	For any element $u=(u_p)\in O_A^+(V,x)$, we have
	$$
	r(x,u(aL+\nu))=\mu(O^+(V,x)\backslash O^+(V,x)uO_A^+(aL+\nu,x)) \prod\limits_{p<\infty} \mu_p(O^+(aL_p+\nu,x))^{-1}.
	$$	
\end{lemma}
\begin{proof}
We roughly follow the argument of Kneser \cite{Kneser}, modified for our case. Note that $\mu$ is right-invariant and $uO_A^+(aL+\nu,x)u^{-1}=O_A^+(u(aL+\nu),x)$. Hence
	\begin{equation}\label{eqn-lem-measureofclass-proof:1}
		\begin{aligned}
		\mu(O^+(V,x)\backslash O^+(V,x)uO_A^+(aL+\nu,x))&=\mu(O^+(V,x)\backslash O^+(V,x)uO_A^+(aL+\nu,x)u^{-1})\\			
		&=\mu(O^+(V,x)\backslash O^+(V,x)O_A^+(u(aL+\nu),x)).
		\end{aligned}
	\end{equation}
	On the other hand, since
	$$
	\begin{aligned}
	&O^+(V,x)\backslash O^+(V,x)O_A^+(u(aL+\nu),x)\\
\cong&\left(O^+(V,x)\cap O_A^+(u(aL+\nu),x)\right) \backslash O_A^+(u(aL+\nu),x)\\
\cong&\left(O_\infty^+(V,x)/O^+(u(aL+\nu),x)\right)\times \prod\limits_{p<\infty} O^+(u_p(aL_p+\nu),x)		
	\end{aligned}
	$$
	is a fundamental domain and $O^+(u_p(aL_p+\nu),x)=u_pO^+(aL_p+\nu,x)u_p^{-1}$, it follows from the left and right invariance of $\mu_p$ that \eqref{eqn-lem-measureofclass-proof:1} is equal to
	$$
	\mu_\infty \left(O_\infty^+(V,x)/O^+(u(aL+\nu),x)\right)\prod\limits_{p<\infty} \mu_p(O^+(aL_p+\nu,x)).
	$$
	This completes the proof of the lemma.
\end{proof}

\begin{lemma}\label{lem-representation-classes-bijection}
	Let $aL+\nu$ be a coset on $V$ and let $x\in V$ with $Q(x)=n$. Then there are bijections
	$$
	\begin{aligned}
		&\{[x,u(aL+\nu)] : [u]\in O^+(V,x)\backslash O_A^+(V)/O_A^+(aL+\nu), \, x\in u(aL+\nu)\}\\
		&\longleftrightarrow\{[(x,aM+\xi)] : aM+\xi \in \pgen(aL+\nu), \, x\in aM+\xi\} \\
		&\longleftrightarrow\{[(y,aL_i+\nu_i)]: Q(y)=n, \, y\in aL_i+\nu_i , \, 1\le i \le h\},		
	\end{aligned}
	$$
	where $\{aL_i+\nu_i\}_{1\le i \le h}$ is a fixed set of representatives of proper classes of $\pgen(aL+\nu)$.
\end{lemma}
\begin{proof}
	The first bijection follows from the definition of the genus of the representation $(x,aM+\xi)$.
	To construct the second map, let $(x,aM+\xi)$ be a representation with $aM+\xi\in \pgen(aL+\nu)$. Note that $aM+\xi= u(aL_i+\nu_i)$ for some $u\in O^+(V)$ and $1\le i \le h$.
	We define a map $\Phi$ from the second set into the third set by $\Phi([(x,aM+\xi)])=[(ux,aL_i+\nu_i)]$. One may check that $\Phi$ is well-defined and is a bijection.
	This proves the lemma.
\end{proof}
\begin{lemma}\label{lem-measure-repnum}
	Let $aL+\nu$ be a coset on $V$ and let $x\in V$ be such that $Q(x)=n$. We have 
	$$
	r(n,aL+\nu)=\frac{\lambda_\infty(O_\infty^+(V))}{\mu_\infty(O_\infty^+(V,x))} \left(\sum r(y,aL+\nu)\right) \lambda_\infty (O_\infty^+(V)/O^+(aL+\nu))^{-1},
	$$
	where the sum runs over a system of representatives of the classes of representations $(y,aL+\nu)$ of $n$ by $aL+\nu$.
	Moreover, taking $\mathfrak{h}$ to be either $\pspn(aL+\nu)$ or $\pgen(aL+\nu)$, we have
	$$
	r(n,\mathfrak{h})=\frac{\lambda_\infty(O_\infty^+(V))}{\mu_\infty(O_\infty^+(V,x))}\left(\sum\limits_{aM+\xi\in \mathfrak{h}} r(x,aM+\xi) \right) \left(\sum\limits_{\pcls(aM+\xi) \in \mathfrak{h}/_\sim} \lambda_\infty (O_\infty^+(V)/O^+(aM+\xi)) \right)^{-1},
	$$
	where the first summation runs over a system of representatives of the classes of representations $[(x,aM+\xi)]$ with $aM+\xi \in \mathfrak{h}$.
	Furthermore, the denominator
	$$
	\sum\limits_{\pcls(aM+\xi) \in \pspn(aL+\nu)/_\sim} \lambda_\infty (O_\infty^+(V)/O^+(aM+\xi))= \lambda_\infty (O_\infty^+(V)) \cdot \text{Mass}(\pspn(aL+\nu))
	$$
	has the same value for all proper spinor genera in $\pgen(aL+\nu)$, 
where 

\end{lemma}
\begin{proof}
	Note that the group $O^+(aL+\nu)$ acts on the set $R(n,aL+\nu)=\{y\in aL+\nu : Q(y)=n\}$, and the orbit of $y\in R(n,aL+\nu)$ with respect to this action corresponds to the representation class of $(y,aL+\nu)$. Therefore,
	$$
	r(n,aL+\nu)= \sum\limits_{[y]:\text{orbits}} |\{\sigma y : \sigma \in O^+(aL+\nu)\}| = \sum\limits_{[(y,aL+\nu)]} \frac{o^+(aL+\nu)}{o^+(aL+\nu,y)},
	$$
	where the last sum runs over a system of representatives of the classes of representations $(y,aL+\nu)$ of $n$ by $aL+\nu$.
	Noting that $r(y,aL+\nu)=\frac{\mu_\infty (O_{\infty}^+(V,y))}{o^+(aL+\nu,y)}$ and  $\lambda_\infty (O_\infty^+(V)/O^+(aL+\nu))=\frac{\lambda_\infty (O_\infty^+(V))}{o^+(aL+\nu)}$, we obtain the first equality of the lemma after noting that $\mu_{\infty}(O_{\infty}^+(V,y))$ only depends on $Q(y)=n$.
	
	The moreover part of the lemma when $\mathfrak{h}=\pgen(aL+\nu)$ follows by the second bijection in Lemma \ref{lem-representation-classes-bijection} together with the proof of the first equality of the lemma.
	The proof for the case when $\mathfrak{h}=\pspn(aL+\nu)$ may also be done in the same manner since the second bijection of Lemma \ref{lem-representation-classes-bijection} still holds when restricted to $\pspn(aL+\nu)$ by the same argument.
	
	Finally, applying a similar argument as in Lemma \ref{lem-measure-class}, one may show that for any $u\in O_A^+(V)$,
	$$
	\begin{aligned}
	\sum\limits_{\pcls(aM+\xi) \in \pspn(u(aL+\nu))/_\sim}& \lambda_\infty (O_\infty^+(V)/O^+(aM+\xi))\\
	&=\lambda(O^+(V)\backslash O^+(V)O_A'(V)uO_A(aL+\nu))\prod\limits_{p<\infty} \lambda_p(O^+(aL_p+\nu))^{-1}\\
	&=\lambda(O^+(V)\backslash O^+(V)O_A'(V)O_A(aL+\nu))\prod\limits_{p<\infty} \lambda_p(O^+(aL_p+\nu))^{-1}
	\end{aligned}
	$$
	The last equality holds because $O^+(V)O_A'(V)uO_A^+(aL+\nu)=O^+(V)O_A'(V)O_A^+(aL+\nu)u$ since $O_A'(V)$ contains the commutator group of $O_A^+(V)$, and $\lambda$ is invariant under right multiplication 	(see also \cite[Theorem 2.4]{LSun}).
	This proves the furthermore part of the lemma.	
\end{proof}

Now we are ready to prove Theorem \ref{thm-repnumofspinor}, which relates the representation numbers for different spinor genera in the genus of a shifted lattice $aL+\nu$ for a ternary lattice $L$
\begin{theorem}\label{thm-repnumofspinor}
	Let $t$ be a square-free positive integer, $aL+\nu$ be a ternary lattice coset, and set $E=\Q(\sqrt{-td_L})$. Then we have\\
\begin{enumerate}[leftmargin=*,label=\rm{(\arabic*)}]
\item If $\theta(O^+(aL_p+\nu)) \not\subseteq N_{E_\mathfrak{p}/\Q_p}(E_\mathfrak{p}^\times)$ for a prime $p$ $(\mathfrak{p}\mid p)$, then
$$
r(tm^2,\pspn(aL+\nu))=r(tm^2,\pspn(aM+\xi))
$$ 
for all $m\in\N$ and $aM+\xi\in\pgen(aL+\nu)$.
\item  If $\theta(O^+(aL_p+\nu)) \subseteq N_{E_\mathfrak{p}/\Q_p}(E_\mathfrak{p}^\times)$ for all primes $p$ $(\mathfrak{p}\mid p)$, then the genus splits into two half-genera and 
$$
r(tm^2,\pspn(aL+\nu))=r(tm^2,\pspn(aM+\xi))
$$ 
for all $m\in\N$ and $aM+\xi$ in the same half-genus of $aL+\nu$ with respect to $t$.
\end{enumerate}
\end{theorem}
\begin{proof}
Let $V$ be the ternary quadratic space containing $aL+\nu$ and for $n=tm^2$ we let $x\in aL+\nu$ be a vector with $Q(x)=n$. We furthermore let $W$ denote the subspace orthogonal to $x$ in $V$, that is, $V=\Q x\perp W$.  Recall that the proper spinor genera from the $\pgen(aL+\nu)$ correspond to the double cosets $O^+(V)O_A'(V)uO_A^+(aL+\nu)$ with $u\in O_A^+(V)$, and note that by Lemmas \ref{lem-measure-class} and \ref{lem-measure-repnum}, the contribution of the genus of $(x,aL+\nu)$ to $r(n,\pspn(u(aL+\nu)))$ is $\mu_{\infty}(O_\infty^+(V,x))^{-1}\text{Mass}(\pspn(aL+\nu))^{-1}$ times
\begin{equation}\label{eqn-contribution-to-pspn}
\mu(O^+(V,x)\backslash O_A^+(V,x)\cap  O^+(V)O_A'(V)uO_A^+(aL+\nu)) \prod\limits_{p<\infty} \mu_p(O^+(aL_p+\nu,x))^{-1}.	
\end{equation}
Since $O_A'(V)$ contains the commutator group of $O_A^+(V)$, $u$ can be extracted to the right. 
If $u\in O_A^+(V,x)$, then by the right invariance of $\mu$, \eqref{eqn-contribution-to-pspn} is independent of $u$.
On the other hand, note that $O^+(V)O_A'(V)uO_A^+(aL+\nu)=O^+(V)O_A'(V)vO_A^+(aL+\nu)$ for some $v\in O_A^+(V,x)$ if and only if 
\begin{equation}\label{eqn-Adele-cond-halfgenera}
u\in O^+(V)O_A'(V)O_A^+(V,x)O_A^+(aL+\nu),
\end{equation}
and by Proposition \ref{prop-spinornormmap}, noting that $O_A^+(V,x)=O_A^+(W)$, it is equivalent to
$$
\theta(u)\in \Q^\times N_{E/\Q}(I_E)\prod\limits_{p\in\Omega} \theta(O^+(aL_p+\nu)).
$$
We naturally split the index giving the number of proper spinor genera in $\gen^+(aL+\nu)$ from  Proposition \ref{prop-spinornormmap} into 
\begin{multline}\label{eqn:indexsplit}
\left[I_{\Q}:\Q^\times\prod_{p\in\Omega} \theta\left(O^+\left(aL_p+\nu\right)\right)\right] = \left[I_\Q:\Q^\times N_{E/\Q}(I_E)\prod\limits_{p\in\Omega} \theta(O^+(aL_p+\nu))\right]\\
\times\left[\Q^\times N_{E/\Q}(I_E)\prod\limits_{p\in\Omega} \theta(O^+(aL_p+\nu)):\Q^{\times}\prod_{p\in\Omega} \theta\left(O^+\left(aL_p+\nu\right)\right)\right].
\end{multline}
We claim that the first factor in \eqref{eqn:indexsplit} is always either $1$ or $2$, and these precisely correspond to the cases (1) and (2) of the theorem, respectively. To show this, we first evaluate the first factor in \eqref{eqn:indexsplit}. Note that $[I_\Q:\Q^\times N_{E/\Q}(I_E)]\le 2$ by \cite[65:21]{OMBook} and this index is equal to $2$ if and only if $-nd_L\notin (\Q^\times)^2$. In particular, the first factor in \eqref{eqn:indexsplit} is at most $2$. 
Furthermore, we have $[I_\Q:\Q^\times N_{E/\Q}(I_E)\prod\limits_{p\in\Omega} \theta(O^+(aL_p+\nu))]=2$ if and only if
\begin{gather}
\prod\limits_{p\in\Omega} \theta(O^+(aL_p+\nu))\subseteq\Q^\times N_{E/\Q}(I_E) \quad \text{and}\quad -nd_L\notin(\Q^\times)^2 \nonumber \\
	\Leftrightarrow \prod\limits_{p\in\Omega} \theta(O^+(aL_p+\nu))\subseteq N_{E/\Q}(I_E) \quad \text{and}\quad -nd_L\notin(\Q^\times)^2 \label{condition-splitting-integers}
\end{gather}
Only the assertion for the last $\Rightarrow$ need some explanation. For fixed $p\in\Omega$, let $x(p)\in \theta(O^+(aL_p+\nu))$, and consider $i=(i_q)\in \prod\limits_{q\in\Omega} \theta(O^+(aL_q+\nu))$ such that $i_p=x(p)$ and $i_q=1$ for any $q\neq p$. Then there exist $b\in \Q^\times$ such that $b\cdot i\in N_{E/\Q}(I_E)$. Since $b$ is a local norm at every spot $q\neq p$, the Hilbert symbol $\left(\frac{b,-td_L}{q}\right)=1$ for any $q\neq p$. By the Hilbert reciprocity law, we have $\left(\frac{b,-td_L}{p}\right)=1$, hence $b$ is a local norm at $p$. Since $b\cdot i_p$ is a local norm at $p$, $x(p)=i_p$ is also a local norm at $p$, hence proving the assertion. 

The condition in \eqref{condition-splitting-integers} precisely splits into the two cases (1) and (2) given in the theorem.  If \eqref{condition-splitting-integers} holds, then we are in case (2) and the first factor on the right-hand side of \eqref{eqn:indexsplit}, which is precisely the index of the group on the right-hand side of \eqref{eqn-Adele-cond-halfgenera} in $O_A^+(V)$ by Proposition \ref{prop-spinornormmap}, is $2$. Hence $\pgen(aL+\nu)$ is divided into two half-genera, both containing the same number of proper spinor genera given by the second factor in \eqref{eqn:indexsplit}, which can be rewritten as 
$$
\left[O^+(V)O_A'(V)O_A^+(V,x)O_A^+(aL+\nu):O^+(V)O_A'(V)O_A^+(aL+\nu)\right],
$$
in such a way that the the genus of the representation $(x,aL+\nu)$ makes the same contribution to $r(n,\pspn(aM+\xi))$ for any coset $aM+\xi$ in the same half-genus of $aL+\nu$.

Otherwise, if \eqref{condition-splitting-integers} does not hold, then we are in case (1) and the the genus of $(x,aL+\nu)$ makes the same contribution to $r(n,\pspn(u(aL+\nu)))$ for any $u\in O_A^+(V)$.

Note that the conditions in \eqref{condition-splitting-integers} do not depend on $aL+\nu$ and $x$, but only on $\pgen(aL+\nu)$, $d_L$ and $n$. Hence, as we run through all genera of representations $(x,u(aL+\nu))$ with $u\in O_A^+(V)$, the determination of whether $\pgen(aL+\nu)$ splits into two halves remains the same. Therefore, by Lemmas \ref{lem-representation-classes-bijection} and \ref{lem-measure-repnum}, we conclude the theorem.
\end{proof}
\begin{remark} If the rank of $L$ is greater than $3$, then the orthogonal complement $W$ of a vector $x$ in $V$ is of rank at least $3$. Hence there is a $v\in O_A^+(W) = O_A^+(V,x)$ such that $u\in O_A'(V)v$. Thus the value \eqref{eqn-contribution-to-pspn} is independent of $u\in O_A^+(V)$ so that we have $r(n,\pspn(aM+\xi))$ are the same for any $aM+\xi\in \pgen(aL+\nu)$. This provides a proof for Teterin's statement \cite[Theorem 1 (1)]{Teterin}.
\end{remark}

From the above theorem, we may show the difference of two theta series $\T{\pspn(aL+\nu)}(z)$ and $\T{\pgen(aL+\nu)}(z)$ is in the space $U$, which determines the second piece in the splitting \eqref{eqn-splitting-of-thetaofcosets}.
\begin{theorem}\label{thm-diff-thetaofpspn-thetaofpgen}
	The Fourier coefficients of $\T{\pspn(aL+\nu)}(z)-\T{\pgen(aL+\nu)}(z)$ are supported on 
\[
\bigcup_{t\in \mathcal{T}_{\pgen(aL+\nu)}} t\Z^2
\]
where
$$
	\mathcal{T}_{\pgen(aL+\nu)}:=\left\{t\in\N: t\text{ is square-free, } \prod\limits_{p\in\Omega}\theta\left(O^+(aL_p+\nu)\right) \subseteq N_{\Q(\sqrt{-td_L})/\Q}\left(I_{\Q (\sqrt{-td_L})}\right) \right\}
	$$
is a finite set. Furthermore, we have $\T{\pspn(aL+\nu)}(z)-\T{\pgen(aL+\nu)}(z)\in U$.
\end{theorem}
\begin{proof}
	Note that for any prime $p$ such that $p\nmid 4N_La^2$ (hence $p\nmid d_L$), we have
	$$
	\theta(O^+(aL_p+\nu))=\theta(O^+(L_p))=\Z_p^\times (\Q_p^\times)^2,
	$$
which contains non-square units, and hence $p\nmid t$ for any $t\in \mathcal{T}_{\pgen(aL+\nu)}$. Therefore, the set $\mathcal{T}_{\pgen(aL+\nu)}$ is a finite set, and it follows from the independence of the spinor masses in Lemma \ref{lem-measure-repnum} and the equality in Theorem \ref{thm-repnumofspinor} (1) that the Fourier coefficients of $\T{\pspn(aL+\nu)}(z)-\T{\pgen(aL+\nu)}(z)$ are supported on square classes in $\bigcup_{t\in \mathcal{T}_{\pgen(aL+\nu)}}t\Z^2$, yielding the first claim.

It remains to show that $f(z):=\T{\pspn(aL+\nu)}(z)-\T{\pgen(aL+\nu)}(z)\in U$. Recall that 
	$$
	f(z)\in S_{3/2}(\Gamma_0(4N_La^2)\cap \Gamma_1(a),\chi_{4d_L}) = \bigoplus_{\chi \pmod{a}}  S_{3/2}(\Gamma_0(4N_La^2),\chi\chi_{4d_L}),
	$$
	and write $f(z)=\sum_{\chi \pmod{a}} f_\chi(z)$ for some $f_\chi(z)\in S_{3/2}(\Gamma_0(4N_La^2),\chi\chi_{4d_L})$.
	
	We claim that the Fourier coefficients of $f_\chi(z)$ are also supported on the square classes in $\mathcal{T}_{\pgen(aL+\nu)}$ for each Dirichlet character $\chi$ modulo $a$.
	If we show the claim, then we have
	$$
	f_\chi(z)\in \bigoplus_{t\in \mathcal{T}_{\pgen(aL+\nu)}} U_t(\Gamma_0(4N_La^2),\chi\chi_{4d_L}),
	$$
	which implies the theorem. Let $f(z)=\sum\limits_{n\ge 1} a(n)q^n$ and $f_\chi(z)=\sum\limits_{n\ge 1} a_\chi(n)q^n$. Note that for any $s\in (\Z/a\Z)^\times$, there exists an integer $s_0$ with $s_0\equiv s\pmod{a}$ such that $\gamma_{s_0}=\left(\begin{smallmatrix}
		\ast & \ast\\
		\ast & s_0
	\end{smallmatrix}\right) \in \Gamma_0(4N_La^2)$.
	Using the modularity of $f_{\chi}$, we have
	\begin{equation}\label{eqn-actgammaon-f-fchi}
	f|_{3/2}\gamma_{s_0}(z)=\sum\limits_{\chi \pmod{a}} f_\chi|_{3/2}\gamma_{s_0}(z)=\chi_{4d_L}(s_0)\cdot\sum\limits_{\chi \pmod{a}} \chi(s)f_\chi(z).
	\end{equation}
	On the other hand, by \eqref{defn-thetaofpropergenus}, \eqref{defn-thetaofproperspinorgenus}, \eqref{eqn-theta-transformed-by-Gamma0}, and Lemma \ref{lem-genus-structure}, we have
	\begin{equation}\label{eqn-actgammaon-f-thetaofpspnpgen}
	f|_{3/2}\gamma_{s_0}(z)=\chi_{4d_L}(s_0) \cdot (\T{\pspn(aL+\overline{s_0}\nu)}(z)-\T{\pgen(aL+\overline{s_0}\nu)}(z)),
	\end{equation}
	where $\overline{s_0}$ is an integer which is an inverse of $s_0$ modulo $a$. Note that $\mathcal{T}=\mathcal{T}_{\pgen(aL+\nu)}=\mathcal{T}_{\pgen(aL+\overline{s_0}\nu)}$ since $O^+(aL_p+\nu)=O^+(aL_p+\overline{s_0}\nu)$ for any prime $p$ by Lemma \ref{lem-genus-structure} (1).
	Comparing Fourier coefficients of the right-hand sides of \eqref{eqn-actgammaon-f-fchi} and \eqref{eqn-actgammaon-f-thetaofpspnpgen}, we may conclude that for any positive integer $n$ outside any of the square classes in $\bigcup_{t\in \mathcal{T}} t\Z^2$,	
	$$
	\sum\limits_{\chi \pmod{a}} \chi(s) a_\chi(n) = 0.
	$$
	Since the above equality holds for any $s\in (\Z/a\Z)^\times$, it follows from the orthogonality of the Dirichlet characters modulo $a$ that $a_\chi(n)=0$ for any $\chi$ modulo $a$.
	This proves the claim, hence completing the proof of the theorem. 
\end{proof}

One may observe from the proof of the above theorem that for $aM+\xi\in\pgen(aL+\nu)$, the differences $\T{\pspn(aL+s\nu)}(z)-\T{\pspn(aM+s\xi)}(z)$ for any integer $s$ coprime to the conductor $a$ of $aL+\nu$ share some property. The following corollary describes some relation on their Fourier coefficients.

\begin{corollary}\label{cor-determining-coefficientsofunarytheta}
	Let $t$ be a square-free positive integer and let $s$ be an integer coprime to the conductor $a$ of a coset $aL+\nu$. Let $aM+\xi$ be a coset in $\pgen(aL+\nu)$, and define $a_s(n)$ by
	$$
	r(tn^2,\pspn(aL+s\nu))-r(tn^2,\pspn(aM+s\xi))=a_s(n)\cdot n.
	$$
	If $a_s(n)$ is not identically zero, then $4t \mid 4N_La^2$.
	If $4N_La^2=4t\cdot t'\cdot b^2$ with square-free $t'$, then $a_s(n)$ is defined modulo $b$ and satisfying
	\begin{enumerate}[label={\rm (\arabic*)}]
		\item $a_{s}(nm)=a_{s\bar{m}}(n)\legendre{-4td_L}{m}$ if $(m,4N_La^2)=1$,
		\item $a_s(n)=0$ if $b\mid n$,
	\end{enumerate}
	where $\bar{m}$ is an integer which is an inverse of $m$ modulo $a$.
\end{corollary}
\begin{proof}
	Note that if $\pi_t$ denotes the projection onto $U_t=\oplus_{\chi\pmod{a}} U_t(4N_La^2,\chi\chi_{4d_L})$, then
	$$
	f_s(z):=\sum\limits_{n\ge 1} a_s(n)nq^{tn^2}=\pi_t(\T{\pspn(aL+s\nu)}(z)-\T{\pspn(aM+s\xi)}(z)).
	$$
	Write $f_s(z)=\sum\limits_{\chi \pmod{a}} f_{s,\chi}(z)$ with $f_{s,\chi}(z)=\sum\limits_{n\ge 1} a_{s,\chi}(n)nq^{tn^2}\in U_t(4N_La^2,\chi\chi_{4d_L})$.
	Since the space $U_t(4N_La^2,\chi\chi_{4d_L})$ is spanned by \eqref{UtNchi-is-spanned-by} with $N\mapsto4N_La^2$ and $\chi\mapsto\chi\chi_{4d_L}$, we have $4t\mid 4N_La^2$ and $a_{s,\chi}(n)$ is defined modulo $b$, hence so is $a_s(n)$. Furthermore,  we have
	$$
	a_{s,\chi}(nm)=a_{s,\chi}(n)\chi(m) \legendre{-4td_L}{m}
	$$
	for any integer $m$ with $(m,4N_La^2)=1$.
	On the other hand, following the same argument used in the proof of Theorem \ref{thm-diff-thetaofpspn-thetaofpgen} to obtain \eqref{eqn-actgammaon-f-fchi} and \eqref{eqn-actgammaon-f-thetaofpspnpgen}, we have for any integer $m$ with $(m,4N_La^2)=1$ that (noting that $(m,4d_L)=1$, hence $\chi_{4d_L}(m)\neq0$)
	$$
	\sum\limits_{\chi\pmod{a}} \chi(m) f_{s,\chi}(z) = f_{s\bar{m}}(z)\text{, hence} 	\sum\limits_{\chi\pmod{a}}\chi(m) a_{s,\chi}(n)= a_{s\bar{m}}(n) \text{ for any }n\in\N.
	$$
	Therefore, for any integer $m$ with $(m,4N_La^2)=1$, we have
	$$
	a_{s}(nm)=\sum\limits_{\chi \pmod{a}} a_{s,\chi}(nm)= \sum\limits_{\chi \pmod{a}}a_{s,\chi}(n)\chi(m)  \legendre{-4td_L}{m} =a_{s\bar{m}}(n)\legendre{-4td_L}{m}.
	$$
	This proves $(1)$. To prove $(2)$, we note that as in the proof of Corollary \ref{cor-theta-pgen-gen-equal}, there is a prime $p$ such that $p\nmid 4N_La^2$, $p\equiv \pm1 \pmod{a}$, and $\legendre{-4td_L}{p}=-1$. Also, by Lemma \ref{lem-genus-structure}, we have $a_s(n)=a_{s\bar{p}}(n)$ for any $p\equiv \pm 1\pmod{a}$.
	If $b\mid n$, then since $a_s(n)$ is defined modulo $b$, we have
	$$
	a_s(n)=a_s(np)=a_{s\bar{p}}(n)\legendre{-4td_L}{p}=-a_{s}(n),
	$$
	and hence $a_s(n)=0$. This completes the proof of the corollary.
\end{proof}

\begin{remark}\label{rem:Teterin-counterexample}
	We remark a failure of the statement of Teterin \cite[Theorem 1 (2)]{Teterin} by giving a counter-example.
	Let $L=\Z e_1 + \Z e_2 + \Z e_3$ be a ternary lattice with a basis $\{e_i\}$ whose corresponding Gram matrix is a diagonal matrix $\diag(1,1,1)$.
	Put $a=12$ and $\nu= 5(e_1+e_2+e_3)\in L$. 
	According to \cite[Theorem 1 (2)]{Teterin}, in order for a positive integer $m\in t\Z^2$ with a square-free $t$ to satisfy 
	$$
	r(m,\pspn(aL+\nu))-r(m,\pgen(aL+\nu))\neq 0,
	$$ 
	one should necessarily have $t|N_L$. Since $N_L=1$, the only candidate is $t=1$.
	
	However, this turns out to be wrong since Haensch and the first author verified in \cite{HaenschKane} that
	$$
	\T{\pspn(aL+\nu)}(z)=\T{\pgen(aL+\nu)}(z)-\frac{1}{8}\sum_{r\in\Z \atop r\equiv 1 \pmod{4}} rq^{3r^2}
	$$
	by explicitly constructing representatives of proper classes in $\pspn(aL+\nu)$ and $\pgen(aL+\nu)$, respectively, and by checking that the first finitely many (up to a certain number coming from the so-called ``valance-formula") Fourier coefficients of the both sides are equal. We refer readers to \cite[Lemma 5.1 and (5.2)]{HaenschKane} for details on this example.
\end{remark}

\section{Comparison of the theta series of ternary lattice cosets in the same spinor genus}\label{section-thetainthesamespinorgenus}

\subsection{Deficiency between theta functions of proper classes and proper spinor genera}
Let $p$ be a prime number such that $p\nmid 4N_La^2$.
Let $Z_p(aL+\nu)$ be the set of cosets $aK+\mu\in\pgen(aL+\nu)$ such that
$$
\Z[1/p](aL)+\nu = \Z[1/p](aK)+\mu,
$$
equivalently,
\begin{equation}\label{eqn:Zpequivalent}
aL_q+\nu = aK_q+\mu \text{ for all } q\neq p \quad \text{and}\quad L_p\cong K_p.
\end{equation}
In the case of lattices, the set $Z_p(aL)$ coincides with the set $Z(aL,p)$ defined in \cite{Schulze-PillotGraph}. Note that one may show from the definition that 
\begin{equation}\label{eqn-inclusion-neighborhoodinZp}
	\{aM+p\xi : aM+\xi \in R_p(aL+\nu)\} \subseteq Z_p(aL+\nu).
\end{equation}
In particular, we have $R_p(aL+\nu)\subseteq Z_p(aL+\nu)$ if $p\equiv 1 \pmod{a}$. 
\begin{lemma}\label{lem-conseq-of-strongapproximation}
	Let $aL+\nu$ be a ternary coset with conductor $a$, and let $p$ be a prime number such that $p\nmid 4N_La^2$.
	Then for any $aM+\xi\in \pspn(aL+\nu)$, there exists a coset $aK+\mu\in Z_p(aL+\nu)$ such that $aK+\mu \in \pcls(aM+\xi)$.
\end{lemma}
\begin{proof}
	Let $aM+\xi$ be a coset in $\pspn(aL+\nu)$. Then there exist a $\sigma\in O^+(V)$ and $\Sigma=(\Sigma_q)\in O_A'(V)$ such that
	\begin{equation}\label{eqn:sigmaqSigmaq}
	aL_q+\nu = \sigma_q \Sigma_q (aM_q+\xi) \quad \text{for any prime $q$}.
	\end{equation}
Following \cite[Section 101]{OMBook}, we choose a basis $x_1,\dots,x_n$ of $V=\Q L$ and for 
\[
x=\sum_{j=1}^n a_j x_j \in V
\]
 we define the norm $\lVert x \rVert_q:=\max_{1\leq j\leq n} \lVert a_j\rVert_q$ for any prime $q$.  Let $J:=\Z x_1 + \cdots +\Z x_n$ be the lattice of elements with $\lVert x\rVert_q\leq 1$ for all $q$ and set $S:=\{q:\text{prime} \mid q\neq p\}$.
	Let $T$ be a finite set of prime numbers not containing $p$ such that
	$$
	aL_q+\nu=aL_q, \ aM_q+\xi = aM_q, \text{ and } \sigma_q^{-1}(aL_q)=J_q=aM_q \quad \text{for any $q\notin T\cup \{p\}$}.
	$$
	Note that $S$ is an indefinite set of spots since $V_p$ is isotropic. Therefore, by the strong approximation for rotation \cite[104:4]{OMBook}, there exist a $\rho\in O'(V)$ such that
	$$
	\begin{cases}
		\lVert \rho_q\rVert_q=1 & \text{if }q\notin  T\cup \{p\},\\		
		\lVert \rho_q-\Sigma_q \rVert_q<\varepsilon & \text{if } q\in T,
	\end{cases}
	$$
	where $\varepsilon>0$ is chosen small enough such that $\Sigma_q(M_q)=\rho_q(M_q)$ and $\sigma_q\xi - \Sigma_q \xi \in \Sigma_q(aM_q)$ for any $q\in T$.
	Now we set $aK+\mu=\sigma\rho(aM+\xi)\in\cls^+(aM+\xi)$. Then from the constructions, noting that $\rho_qJ_q=J_q$ for all $q\not\in T\cup\{p\}$ by \cite[101:4]{OMBook}, and using \eqref{eqn:sigmaqSigmaq}, 
	$$
	aK_q+\mu=
	\begin{cases}
		\sigma_q\rho_q(aM_q+\xi)=\sigma_q\rho_q(aM_q)=\sigma_q(\rho_qJ_q)=\sigma_q(J_q)=aL_q=aL_q+\nu & \text{if } q\notin T\cup \{p\},\\
		\sigma_q\rho_q(aM_q+\xi)=\sigma_q(\Sigma_q(aM_q)+\rho_q\xi)=\sigma_q(\Sigma_q(aM_q)+\Sigma_q\xi)=aL_q+\nu & \text{if } q\in T,\\
		K_p\cong L_p &\text{if } q=p.
	\end{cases}
	$$
	Therefore, $aK+\mu \in Z_p(aL+\nu)$ by \eqref{eqn:Zpequivalent}, and this proves the lemma.
\end{proof}

Finally, we are ready to prove the following theorem that determines the third piece in the splitting \eqref{eqn-splitting-of-thetaofcosets}, the cusp form which is orthogonal to the unary theta functions.
\begin{theorem}\label{thm-Theta-same-spn}
	Let $aM+\xi \in \pspn(aL+\nu)$. Then we have
	$$
	\T{aL+\nu}(z)-\T{aM+\xi}(z)\in U^\perp.
	$$
	Moreover, we have $\T{aL+\nu}(z)-\T{\pspn(aL+\nu)}(z)\in U^\perp$.
\end{theorem}
\begin{proof}
	The second assertion follows directly from \eqref{defn-thetaofproperspinorgenus} once we prove the first assertion.
	The proof for the first assertion will follow an argument similar to \cite[Satz 4]{Schulze-PillotTernaryTheta}. 
	Let $p$ be a prime number such that $p\equiv 1 \pmod{8N_La}$.
	For a square-free positive integer $t$, let $\pi_t$ denote the projection onto $U_t=\oplus_{\chi\pmod{a}} U_t(4N_La^2,\chi\chi_{4d_L})$. 
If $t\nmid N_La^2$, then Corollary \ref{cor-determining-coefficientsofunarytheta} implies that $\pi_t\left(\T{aL+\nu}-\T{aM+\xi}\right)=0$. Now suppose that $t\mid N_La^2$. Note that $\bar{p}\equiv 1\pmod{a}$ so that $aL+\bar{p}\nu = aL+\nu$ and $\pgen(aL+\bar{p}\nu)=\pgen(aL+\nu)$.
	Moreover, since the projection operators $\pi_t$ commute with the Hecke operator $T(p^2)$ and $U_t$ is an eigenspace under $T(p^2)$ by Proposition \ref{prop-U_t,chi-eigenspace}, we conclude from Theorem \ref{thm-Hecke-nbd} that
	\begin{equation}
		\begin{aligned}
			\sum\limits_{aK+\mu\in R_p(aL+\nu)} \pi_t(\T{aK+\mu})&=\pi_t(\T{aL+\nu}|T(p^2))\\
						&=\pi_t(\T{aL+\nu})|T(p^2)=(p+1)\pi_t(\T{aL+\nu}).
		\end{aligned}
	\end{equation}
	Here we used the fact that $\chi(p)=1$ since $p\equiv 1\pmod{a}$ and $\legendre{-td_L}{p}=1$ by quadratic reciprocity and $p\equiv 1\pmod{8N_La}$ in the last equality. Since $|R_p(aL+\nu)|=p+1$, we have
	\begin{equation}\label{eqn-thm-Theta-same-spn-proof:1}
		\sum\limits_{aK+\mu \in R_p(aL+\nu)} \pi_t(\T{aK+\mu}-\T{aL+\nu})=0.
	\end{equation}
	We claim that $\pi_t(\T{aJ+\lambda}-\T{aL+\nu})=0$ for any cosets $aJ+\lambda \in Z_p(aL+\nu)$.
	For any $aJ+\lambda \in Z_p(aL+\nu)$, let 
	$$
	\pi_t(\T{aJ+\lambda}(z))=\sum\limits_{m=1}^\infty r''_{aJ+\lambda}(m) q^{tm^2}.
	$$
The claim is then equivalent to showing that 
\begin{equation}\label{eqn:UtcoeffequalZp}
r''_{aJ+\lambda}(m)=r''_{aL+\nu}(m)
\end{equation}
for all $aJ+\lambda\in Z_p(aL+\nu)$ and all $m$. We first show \eqref{eqn:UtcoeffequalZp} for $aJ+\lambda\in R_p(aL+\nu)\subseteq Z_p(aL+\nu)$.  Since any $aJ+\lambda\in R_p(aL+\nu)$ is contained in $\pgen(aL+\nu)$ and the theta function only depends on the choice of class in the genus (of which there are only finitely many), we may assume that $aL+\nu$ is chosen so that 
	$$
	\text{Re}(r''_{aL+\nu}(m)) = \min \{\text{Re}(r''_{aJ+\lambda}(m)) : aJ+\lambda \in R_p(aL+\nu)\}.
	$$
We see that the real part of the $m$-th coefficient of each term in \eqref{eqn-thm-Theta-same-spn-proof:1} is non-negative and the coefficients sum to zero, so $\text{Re}(r''_{aJ+\lambda}(m))=\text{Re}(r''_{aL+\nu}(m))$ for any $aJ+\lambda \in R_p(aL+\nu)$. Making the same argument with the imaginary part, we conclude that $r''_{aJ+\lambda}(m))=r''_{aL+\nu}(m)$ for any $aJ+\lambda \in R_p(aL+\nu)$, giving \eqref{eqn:UtcoeffequalZp} for any $aJ+\lambda\in R_{p}(aL+\nu)$.

To show \eqref{eqn:UtcoeffequalZp} for all $aJ+\lambda\in Z_p(aL+\nu)$, we claim that for any $aJ+\lambda\in Z_p(aL+\nu)$, there is a chain of cosets 
	\begin{equation}\label{eqn:chain}
		aL+\nu=aK_0+\mu_0,\ aK_1+\mu_1, \ldots, aK_n+\mu_n=aJ+\lambda
	\end{equation}
	such that $aK_{i}+\mu_{i}\in R_p(aK_{i-1}+\mu_{i-1})$ for any $1\le i \le n$, which immediately implies \eqref{eqn:UtcoeffequalZp} from the claim for $R_p(aL+\nu)$ used inductively. To see that such a chain exists, we note from the local theory of lattices (cf. \cite[82:23]{OMBook}) that there exists a basis $\{e_1,e_2,e_3\}$ of $L_p$ and $n\in\Z_{\geq 0}$ such that $Q(e_2)=Q(e_3)=0$, $B(e_1,e_2)=B(e_1,e_3)=0$, $B(e_2,e_3)=1$, and $\{e_1,p^{-n}e_2,p^{n}e_3\}$ is a basis of $J_p$. For $0\le i \le n$, taking $aK_i+\mu_i$ to be the cosets on the space $\Q L$ satisfying
	\[
	a(K_i)_q+\mu_i = aL_q+\nu \text{ for all }q\neq p  \quad \text{and} \quad (K_i)_p=\Z_p e_1+ \Z_p p^{-i}e_2 + \Z_p p^i e_3,
	\]
	one may check that they satisfy desired properties from the definitions of $Z_p(aL+\nu)$ and $aK_i+\mu_i\in R_p(aK_{i-1}+\mu_{i-1})$. Therefore, we conclude by induction that \eqref{eqn:UtcoeffequalZp} holds for any $aJ+\lambda\in Z_p(aL+\nu)$.
	
	Now for any $aM+\xi\in \pspn(aL+\nu)$, by Lemma \ref{lem-conseq-of-strongapproximation}, there is a $aK+\mu\in Z_p(aL+\nu)$ such that $aK+\mu \in \pcls(aM+\xi)$.
	Since $\T{aM+\xi}(z) = \T{aK+\mu}(z)$, we have
	$$
	\pi_t(\T{aM+\xi}(z)-\T{aL+\nu}(z))=\pi_t(\T{aK+\mu}(z)-\T{aL+\nu}(z))=0
	$$
	for all $t\mid 4N_La^2$. Therefore, we may conclude that $\T{aM+\xi}(z)-\T{aL+\nu}(z)\in U^\perp$.
\end{proof}

\subsection{An algorithm for computing proper class representatives of proper (spinor) genera} \label{subsection-graph}
In this subsection, we are interested in constructing an algorithm that returns a complete set of representatives of proper classes of $\pspn(aL+\nu)$ (hence, that of $\pgen(aL+\nu)$). In principle, one can iteratively find representatives in $Z_p(aL+\nu)$ and then compute the mass $\text{Mass}(\pspn(aL+\nu))$ to determine when all proper classes have been found. However, an independent calculation of $\text{Mass}(\pspn(aL+\nu))$ (without knowing the set of representatives) is needed for such a construction, so we instead design an algorithm to find the complete set of representatives without computing the mass.
 
Throughout this subsection, let $p$ be a prime number with $p\nmid 4N_La^2$ and $p\equiv 1\pmod{a}$ so that $R_p(aL+\nu)\subseteq Z_p(aL+\nu)$.
Consider the (undirected) graph $X(aL+\nu:p)$ whose vertices consist of the lattice cosets in the set $Z_p(aL+\nu)$.
Two lattice cosets $aK_1+\mu_1$ and $aK_2+\mu_2$ are connected by an edge if and only if one is a $p$-neighborhood of the other (hence vice versa because $p\equiv 1\pmod{a}$). 
Then, as in the lattice case, the graph $X(aL+\nu:p)$ is connected due to the existence of chains, as proven in \eqref{eqn:chain}. Furthermore, it is known (see the discussion at the end of \cite[Section 1]{Schulze-PillotAlgorithm}) that for ternary lattices (i.e., the $a=1$ case), it is a tree. Note that by the definition of $p$-neighborhoods in Section \ref{subsection-p-nbd-cosets}, if $aK_1+\mu_1$ and $aK_2+\mu_2$ are connected in $X(aL+\nu:p)$, then $K_1$ and $K_2$ are connected in $X(L;p)$, so $X(aL+\nu;p)$ is also a tree for $a>1$ in the ternary case.

Moreover, one may show by following a similar argument in \cite{BenhamHsia} that the number $g^+(aL+\nu:p)$ of proper spinor genera represented by $X(aL+\nu:p)$ is at most two, and 
\[
g^+(aL+\nu:p)=1  \quad \text{if and only if} \quad  j(p)\in \Q^\times \prod\limits_{p\in \Omega} \theta(O^+(aL_p+\nu)),
\]
where $j(p)=(j_q)_{q\in\Omega}\in I_\Q$ is the id{\`e}le defined by $j_p=p$ and $j_q=1$ for any $q\in \Omega\setminus\{p\}$.

Assume that we have found $p$ such that $g^+(aL+\nu:p)=1$. We now start finding the vertices of the graph $X(aL+\nu:p)$ to construct a complete set of representatives of proper classes of $\pspn(aL+\nu)$, going through the following algorithm:
\begin{enumerate}
	\item[Step $0$:] Start by taking the set of a vertex $S=S_{\text{new}}^{(0)}:=\{aL+\nu\}$, and put $S_{\text{new}}=S_{\text{new}}^{(0)}$. 
	\item[($\ast$)] Let $i=1$, and repeat the following Step $i$ until $S_{\text{new}}^{(i-1)}=\emptyset$.
	\item[Step $i$:] 
	\item Find $S^{(i)}:=\cup_{aK+\mu\in S_{\text{new}}^{(i-1)}} R_p(aK+\mu)$ by constructing $p$-neighborhoods. 
	\item Find $S_{\text{new}}^{(i)}:=\{aK+\mu \in S^{(i)} : aK+\mu \notin \pcls(aM+\xi) \text{ for all } aM+\xi \in S\}$.
	\item Update $S$ with $S \cup S_{\text{new}}^{(i)}$, and $i$ with $i+1$.
\end{enumerate}
It is clear that this algorithm terminates since the set $S$ consists of inequivalent lattice cosets in $\pspn(aL+\nu)$ by its construction. We claim that $S$ form a complete set of representatives. 

Note that the algorithm returns a finite subtree of $X'(aL+\nu:p)\subseteq X(aL+\nu:p)$ given by iteratively adding nodes of $X(aL+\nu:p)$ of depth $i$ (with root $aL+\nu$) which are connected to nodes of $X'(aL+\nu:p)$ of depth $i-1$ and are not in the same proper class as any node in $X'(aL+\nu:p)$ with depth $<i$. Although $X(aL+\nu:p)$ is connected and contains a representative of every proper class in $\pspn(aL+\nu)$ by Lemma \ref{lem-conseq-of-strongapproximation}, it is not immediately clear that every proper class appears in $X'(aL+\nu:p)$ because our trimming of the tree may have made the classes disconnected. However, we claim that a representative of every such class appears in $X'(aL+\nu:p)$, which is equivalent to showing that $S$ contains a full set of reprentatives.

Let $aL_i+\mu_i\in Z_p(aL+\nu)$ with $i=1,2$ be two lattice cosets which are isometric to each other, say $aL_2+\nu_2 = \sigma(aL_1+\nu_1)$ for some $\sigma\in O^+(V)$. If $\{aK_j+\mu_j\}_{1\le j \le p+1}$ are $p$-neighborhoods of $aL_1+\nu_1$, then one may easily observe from the definition that $\{\sigma(aK_j+\mu_j)\}_{1\le j \le p+1}$ are $p$-neighborhoods of $aL_2+\nu_2=\sigma(aL_1+\nu_1)$. If we take a node $aK+\mu\in X(aL+\nu:p)$ of minimal depth (say $i$) such that no lattice coset in the same proper class as $aK+\mu$ is contained in $X'(aL+\nu:p)$, then let $aK'+\mu'$ be the parent of $aK+\mu$ (with depth $i-1$). Since $aK'+\mu'$ has smaller depth, by minimality of $i$ there must be $aK''+\mu''\in X'(aL+\nu:p)$ in the same proper class as $aK'+\mu'$, but then if $\sigma(aK'+\mu')=aK''+\mu''$, then $\sigma(aK+\mu)$ is a neighbor of $aK''+\mu''$, and when the algorithm finds $aK''+\mu''$, either $\sigma(aK+\mu)$ is the parent of $aK''+\mu''$ (in which case $\sigma(aK+\mu)\in X'(aL+\nu:p)$), or the algorithm would check $\sigma(aK+\mu)$ at the following step. This contradicts the assumption that no representative of the class of $aK+\mu$ is contained in $X'(aL+\nu:p)$.

 To extend this algorithm to obtain a complete set of representatives for $\pgen(aL+\nu)$, note that if $\pspn(aL+\nu)\subsetneq \pgen(aL+\nu)$, then one can obtain any other proper spinor genera in $\pgen(aL+\nu)$ with $p$-neighborhoods of $aL+\nu$ by choosing a prime $p$ carefully; this is possible since the ``spinor linkage theorem" can analogously be extended to lattice cosets (see \cite[Theorem 2 and Remark]{BenhamHsia}).

\end{document}